\pdfoutput=1
\RequirePackage{ifpdf}
\ifpdf 
\documentclass[pdftex]{sigma}
\else
\documentclass{sigma}
\fi

\usepackage[all]{xy}

\newcommand{\bC}{\mathbb{C}}
\newcommand{\bF}{\mathbb{F}}
\newcommand{\bG}{\mathbb{G}}
\newcommand{\bM}{\mathbb{M}}
\newcommand{\bQ}{\mathbb{Q}}
\newcommand{\bR}{\mathbb{R}}
\newcommand{\bZ}{\mathbb{Z}}
\newcommand{\unit}{\mathbf{1}}

\newcommand{\cY}{\mathcal{Y}}

\newcommand{\fm}{\mathfrak{m}}
\newcommand{\fsl}{\mathfrak{sl}}
\newcommand{\ft}{\mathfrak{t}}
\DeclareMathOperator{\Aut}{Aut}
\DeclareMathOperator{\uAut}{\underline{Aut}}
\DeclareMathOperator{\End}{End}
\DeclareMathOperator{\Hom}{Hom}
\DeclareMathOperator{\uHom}{\underline{Hom}}
\DeclareMathOperator{\Res}{Res}
\DeclareMathOperator{\Spec}{Spec}
\DeclareMathOperator{\Tr}{Tr}
\DeclareMathOperator{\wt}{wt}
\newcommand{\simto}{\overset{\sim}{\to}}

\newtheorem{thm}{Theorem}[section]
\newtheorem{conj}[thm]{Conjecture}
\newtheorem{prop}[thm]{Proposition}
\newtheorem{lem}[thm]{Lemma}
\newtheorem{cor}[thm]{Corollary}

\theoremstyle{definition}
\newtheorem{defn}[thm]{Definition}
\newtheorem{rem}[thm]{Remark}

\begin{document}

\allowdisplaybreaks

\newcommand{\arXivNumber}{1710.00737}

\renewcommand{\thefootnote}{}

\renewcommand{\PaperNumber}{030}

\FirstPageHeading

\ShortArticleName{A Self-Dual Integral Form of the Moonshine Module}

\ArticleName{A Self-Dual Integral Form of the Moonshine Module\footnote{This paper is a~contribution to the Special Issue on Moonshine and String Theory. The full collection is available at \href{https://www.emis.de/journals/SIGMA/moonshine.html}{https://www.emis.de/journals/SIGMA/moonshine.html}}}

\Author{Scott CARNAHAN}
\AuthorNameForHeading{S.~Carnahan}
\Address{University of Tsukuba, Japan}
\Email{\href{mailto:carnahan@math.tsukuba.ac.jp}{carnahan@math.tsukuba.ac.jp}}
\URLaddress{\url{http://www.math.tsukuba.ac.jp/~carnahan/}}

\ArticleDates{Received February 13, 2018, in final form April 06, 2019; Published online April 19, 2019}

\Abstract{We construct a self-dual integral form of the moonshine vertex operator algebra, and show that it has symmetries given by the Fischer--Griess monster simple group. The existence of this form resolves the last remaining open assumption in the proof of the modular moonshine conjecture by Borcherds and Ryba. As a corollary, we find that Griess's original 196884-dimensional representation of the monster admits a positive-definite self-dual integral form with monster symmetry.}

\Keywords{moonshine; vertex operator algebra; orbifold; integral form}

\Classification{17B69; 11F22; 20C10; 20C20; 20C34}

\vspace{-3mm}
\tableofcontents

\renewcommand{\thefootnote}{\arabic{footnote}}
\setcounter{footnote}{0}

\section{Introduction}

In this paper, we construct self-dual $R$-forms of the moonshine module vertex operator algebra~\cite{FLM88} over various commutative rings $R$, culminating in the universal case where $R$ is the ring of rational integers $\bZ$. For the $\bZ$-form, we show that the vertex operator algebra has monster symmetry, and is self-dual with respect to an invariant bilinear form that respects the monster symmetry. Base change then gives us self-dual monster-symmetric vertex operator algebras over any commutative ring. Our construction yields the final step in the affirmative resolution of Ryba's modular moonshine conjecture.

Monstrous moonshine began in the 1970s with McKay's observation that the number 196884, namely the $q^1$ coefficient in Klein's modular $j$-invariant, is one more than the dimension of the smallest faithful complex representation of the monster simple group $\bM$. Similar observations by Thompson for higher-order coefficients~\cite{T79} led to what is now called the McKay--Thompson conjecture, asserting the existence of a natural graded faithful complex representation $V = \bigoplus_{n \geq 0} V_n$ of the monster with graded dimension given by
\begin{gather*} \sum_{n \geq 0} \dim V_n q^{n-1} = J(\tau) = j(\tau) - 744.\end{gather*}

The McKay--Thompson conjecture was resolved by the construction of the moonshine mo\-du\-le~$V^\natural$ in~\cite{FLM85}. This construction was then refined in~\cite{FLM88} so that $V^\natural$ is not just a graded vector space, but a vertex operator algebra whose automorphism group is precisely the monster. The vertex operator algebra structure on $V^\natural$ turned out to be essential to Borcherds's resolution of the monstrous moonshine conjecture \cite{B92} for $V^\natural$. This conjecture \cite{CN79} is a substantial refinement of the McKay--Thompson conjecture, given by replacing graded dimensions with graded traces of elements of $\bM$, and asserts that the resulting functions are Hauptmoduln for certain genus zero groups commensurable to ${\rm SL}_2(\bZ)$.

It is natural to ask whether $V^\natural$ can be defined over smaller subrings of $\bC$, and the first progress on this question was given at the end of \cite{FLM88}, where a monster-symmetric self-dual $\bQ$-form is constructed. In~\cite{BR96}, the authors show that the $\bQ$-form can be defined over $\bZ[1/2]$, but find that inversion of 2 is essential to the construction. Finally, an integral form of $V^\natural$ with monster symmetry was constructed in~\cite{DG12}, but part of the construction involves passing to a~monster-invariant sublattice to obtain the full symmetry, and this process generally destroys any control over self-duality. In particular, the inner product may be quite degenerate after reduction modulo a prime. A similar question concerns forms of the 196884-dimensional algebra~$V_2$, which is essentially what Griess used in his initial construction of the monster~\cite{G82}. Conway sketches a construction due to Norton of a monster-symmetric integral form of the 196884-dimensional subalgebra $V_2$ in \cite{C85}, and expresses hope that it may be self-dual. However, the question of extending such a form to the rest of $V^\natural$ has not been addressed. Indeed, remarks in \cite{BR96} suggest that this form would be incompatible with the form that we construct in this paper. We show in this paper that both questions have the best possible answers: there is a~positive definite self-dual integral form of $V^\natural$ with monster symmetry, and its weight 2 subspace is a positive definite self-dual integral form of Griess's algebra (up to a minor change in the multiplication on the Virasoro vector, as is pointed out in \cite{FLM88}), with monster symmetry.

The main motivation for this paper comes from the modular moonshine conjecture, which arose out of experimental observations by Ryba in~\cite{R96} that suggested the existence of a~mod~$p$ version of moonshine. Modular moonshine asserted the existence of certain vertex algebras over~$\bF_p$ attached to elements $g$ of $\bM$ in class $p$A such that Brauer traces of $p$-regular centralizing elements $h$ are given by certain McKay--Thompson series, i.e., graded traces on $V^\natural$. This conjecture, which was somewhat strengthened in~\cite{BR96}, was essentially resolved in~\cite{BR96} and~\cite{B98}, up to some assumptions about the existence of integral versions of various constructions that were known over $\bR$ or $\bC$. Some of these assumptions, such as a description of root spaces of a $\bZ_p$-form of the monster Lie algebra, were resolved in~\cite{B99}. The remaining problem was that the conjecture asserts the existence of a~self-dual $\bZ$-form of $V^\natural$ with $\bM$-symmetry. Since such a~form was not immediately available, the authors of~\cite{BR96} noted that for the original conjecture, one only needs for each prime $p$ dividing the order of~$\bM$, a~self-dual form over the ring $\bZ_p$ of $p$-adic integers. If a~self-dual $\bZ_p$-form (or more generally, a~form over any $p$-adic integer ring) with monster symmetry exists, the arguments given in~\cite{BR96} and~\cite{B98} (combined with some technical details resolved in \cite{B99}) yield a resolution of the conjecture. For odd primes~$p$, it therefore suffices to produce a~self-dual $\bZ[1/2]$-form, and in~\cite{BR96}, the authors showed that the construction in~\cite{FLM88} in fact yields such a form. For the case $p=2$, the problem has remained unambiguously open until now.

In this paper, we give a construction of a self-dual integral form of $V^\natural$ with monster symmetry, which resolves the last open assumption in modular moonshine. The construction requires several very recent developments in the theory of vertex operator algebras. The most notable of these is the cyclic orbifold construction given in~\cite{vEMS}. However, despite the dependence on recent technology, there is an effortless quality to our construction, like we are getting this object for free. In particular, the argument requires a surprisingly small amount of group theoretic input, essentially all compressed into some information about maximal subgroups of $\bM$ in the proof of Lemma~\ref{lem:maximal-subgroups-of-monster}. Aside from this input, the main ingredients are:
\begin{enumerate}\itemsep=0pt
\item The existence and uniqueness of self-dual $R$-forms of abelian intertwining algebras, from generating subalgebras. This is basically given by extending some well-known skew-symmetry and associativity properties of vertex algebras to the abelian intertwining algebra setting.
\item A very recent method for analyzing $V^\natural$ by applying multiple cyclic orbifolds to the Leech lattice vertex operator algebra, due to Abe, Lam, and Yamada. Together with the previous tool, this lets us extract self-dual forms of $V^\natural$ over cyclotomic S-integer rings.
\item Faithfully flat descent. This technique is more than a half-century old and has broad applications, but does not seem to get much attention in the vertex operator algebra world, where the base ring is almost always~$\bC$.
\end{enumerate}

While modular moonshine is now resolved, there are still many related open questions remaining. For example, most of the questions listed at the end of~\cite{BR96} remain unresolved, and the question of what happens in the case $g$ has composite order seems particularly natural.

\section{Cyclic orbifolds over subrings of the complex numbers}

The main goal of this section is to show that the cyclic orbifold construction, as worked out in \cite{vEMS}, is definable over distinguished subrings of~$\bC$. For this purpose, we introduce some commutative algebra technology.

For any complex number $r$, we write $e(r)$ to denote the normalized exponential ${\rm e}^{2\pi {\rm i} r}$. We will also use the following notation conventions, some of which may be unfamiliar to non-specialists in group schemes (see, e.g., \cite{SGA3}):
\begin{enumerate}\itemsep=0pt
\item For any commutative ring $R$ and any $n \in \bZ_{\geq 1}$, we write $\mu_n(R)$ to denote the group $\{ r \in R \,|\, r^n = 1\}$. The functor $\mu_n$ is represented by the affine scheme $\Spec \bZ[x]/\big(x^n-1\big)$.
\item For any commutative ring $R$, we write $\bG_m(R)$ to denote the group $R^\times$ of invertible elements. The functor $\bG_m$ is represented by the affine scheme $\Spec \bZ[x,y]/(xy-1)$, which is also written as the spectrum of the ring $\bZ\big[x,x^{-1}\big]$ of Laurent polynomials.
\item Given commutative group schemes $G$ and $H$, we write $\uHom(G,H)$ for the functor that takes a ring $R$ to the group of natural transformations $G \times \Spec R \to H \times \Spec R$ over $\Spec R$. In this paper, we will only consider the case $G$ is a torus or a lattice of rank $n$, and $H$ is either $\mu_n$ or~$\bG_m$. In particular, $D(\bZ^n) = \uHom(\bZ^n,\bG_m) \cong \bG_m^n$, and $D(\bG_m^n) = \uHom(\bG_m^n, \bG_m) \cong \bZ^n$.
\end{enumerate}

\subsection{Descent and gluing for finite projective modules and bilinear forms}

\begin{defn}
Let $R$ be a commutative ring. We say that an $R$-module $M$ is faithfully flat if for any 3-term complex $M_1 \to M_2 \to M_3$ of $R$-modules, the following conditions are equivalent:
\begin{enumerate}\itemsep=0pt
\item[1)] $M_1 \to M_2 \to M_3$ is exact,
\item[2)] $M \otimes_{R} M_1 \to M \otimes_{R} M_2 \to M \otimes_{R} M_3$ is exact.
\end{enumerate}
Given a homomorphism $f\colon R \to S$ of commutative rings, we say $f$ is faithfully flat if $S$ is faithfully flat as an $R$-module under the induced action.
\end{defn}

See \cite[\href{https://stacks.math.columbia.edu/tag/00H9}{Section~00H9}]{Stacks} for a brief overview of basic properties of flat and faithfully flat modules and ring maps.

\begin{defn} \label{defn:descent-datum-for-modules}
Let $f\colon R \to S$ be a homomorphism of commutative rings. A descent datum for modules with respect to $f$ is a pair $(M, \phi)$, where $M$ is an $S$-module, and $\phi\colon M \otimes_{R} S \to S \otimes_{R} M$ is an $S \otimes_{R} S$-module isomorphism satisfying the ``cocycle condition'', i.e., that the following diagram commutes:
\begin{gather*} \xymatrix{ M \otimes_{R} S \otimes_{R} S \ar[rr]^{\phi_{0,1} = \phi \otimes 1} \ar[rrd]_{\phi_{0,2}} & & S \otimes_{R} M \otimes_{R} S \ar[d]^{\phi_{1,2} = 1 \otimes \phi} \\ & & S \otimes_{R} S \otimes_{R} M,} \end{gather*}
where, if we write $\phi(m \otimes 1) = \sum s_i \otimes m_i$, then the $S \otimes_{R} S \otimes_{R} S$-module homomorphisms $\phi_{i,j}$ are defined by
\begin{gather*}
\phi_{0,1}(m \otimes 1 \otimes 1) = \sum s_i \otimes m_i \otimes 1, \\
\phi_{0,2}(m \otimes 1 \otimes 1) = \sum s_i \otimes 1 \otimes m_i, \\
\phi_{1,2}(1 \otimes m \otimes 1) = \sum 1 \otimes s_i \otimes m_i.
\end{gather*}
A morphism of descent data from $(M, \phi)$ to $(M', \phi')$ is an $S$-module homomorphism $\psi\colon M \to M'$ such that the following diagram commutes:
\begin{gather*} \xymatrix{ M \otimes_{R} S \ar[r]^{\phi} \ar[d]_{\psi \otimes 1} & S \otimes_{R} M \ar[d]^{1 \otimes \psi} \\
M' \otimes_{R} S \ar[r]^{\phi'} & S \otimes_{R} M'. } \end{gather*}
\end{defn}

\begin{thm} \label{thm:descent-for-modules}
Let $f\colon R \to S$ be a homomorphism of commutative rings. If $f$ is faithfully flat, then the category of $R$-modules is equivalent to the category of descent data for modules with respect to $f$. In particular, the functor taking an $R$-module $M$ to the $S$-module $M \otimes_{R} S$ equipped with descent datum $\phi((m \otimes s) \otimes s') = s \otimes (m \otimes s')$ has quasi-inverse given by $(M, \phi) \mapsto \tilde{M} = \{m\in M \,|\, 1 \otimes m = \phi(m \otimes 1)\}$.
\end{thm}
\begin{proof}
See \cite[\href{https://stacks.math.columbia.edu/tag/023N}{Proposition 023N}]{Stacks} or Theorem~4.21 of~\cite{V05} for freely available expositions, or Theorem~1 of~\cite{G59} for the original reference.
\end{proof}

We note that faithful flatness can be replaced by the weaker condition of ``universally injective for modules'', namely that $f \otimes {\rm id}_M\colon M = R \otimes_{R} M \to S \otimes_{R} M$ is injective for all $R$-modules~$M$. The fact that this condition is necessary and sufficient for effectiveness of descent of modules was noted without proof in~\cite{O70} and proved in~\cite{M00}.

\begin{defn}Given a homomorphism $f\colon R \to S$ of commutative rings, a bilinear map with descent datum with respect to $f$ is a quadruple $(\kappa\colon M \times N \to T, \phi_M, \phi_N, \phi_T)$, where $M$, $N$, $T$ are $S$-modules, $\kappa$ is an $S$-bilinear map, and $\phi_M\colon M \otimes_{R} S \to S \otimes_{R} M$, $\phi_N\colon N \otimes_{R} S \to S \otimes_{R} N$, $\phi_T\colon T \otimes_{R} S \to S \otimes_{R} T$ are descent data for modules with respect to $f$, such that the following diagram of $S \otimes_{R} S$-module homomorphisms commutes:
\begin{gather*} \xymatrix{ (M \otimes_{R} S)\otimes_{S \otimes_{R} S} (N \otimes_{R} S) \ar[r]^{\phi_M \otimes \phi_N} \ar[d]_{\bar{\kappa} \otimes 1} & (S \otimes_{R} M) \otimes_{S \otimes_{R} S} (S \otimes_{R} N) \ar[d]^{1 \otimes \bar{\kappa}} \\ T \otimes_{R} S \ar[r]_{\phi_T} & S \otimes T. } \end{gather*}
Here, $\bar{\kappa}\colon M \otimes N \to T$ is the $S$-module map canonically attached to~$\kappa$.
\end{defn}

\begin{lem} \label{lem:descent-for-bilinear-maps}
Given a faithfully flat homomorphism $f\colon R \to S$ of commutative rings, any bi\-li\-near map with descent datum with respect to $f$ is equivalent to the base change of a unique bilinear map of $R$-modules. That is, given $(\kappa\colon M \times N \to T, \phi_M, \phi_N, \phi_T)$, there is a unique $R$-linear map $\tilde{\kappa}\colon \tilde{M} \otimes \tilde{N} \to \tilde{T}$ such that the following diagram commutes:
\begin{gather*} \xymatrix{ \big(\tilde{M} \otimes_{R} S\big) \otimes_S \big(\tilde{N} \otimes_{R} S\big) \ar[r]^-{\tilde{\kappa} \otimes 1} \ar[d]_{\psi_M \otimes \psi_N} & \tilde{T} \otimes_{R} S \ar[d]^{\psi_T} \\
M \otimes N \ar[r]^-{\bar{\kappa}} & T. } \end{gather*}
Here, $\psi_M\colon \tilde{M} \otimes_{R} S \to M$, $\psi_N\colon \tilde{N} \otimes_{R} S \to N$, $\psi_T\colon \tilde{T} \otimes_{R} S \to T$ are the $S$-module isomorphisms induced by the descent data.
\end{lem}
\begin{proof}This amounts to applying the equivalence in Theorem~\ref{thm:descent-for-modules} to the $S$-module map $\bar{\kappa}$.
\end{proof}

\begin{prop} \label{prop:finite-projective-property-descends} Let $f\colon R \to S$ be a faithfully flat homomorphism of commutative rings, and let $M$ be an $R$-module. Then $M$ is a finite projective $($equivalently, flat and finitely presented$)$ $R$-module if and only if $M \otimes_{R} S$ is a finite projective $S$-module.
\end{prop}
\begin{proof}This is given by \cite[Proposition 2.5.2]{EGA4v2}.
\end{proof}

\begin{defn}Given a commutative ring $R$, we define the groupoid of finite projective $R$-modules with bilinear form by setting:
\begin{enumerate}\itemsep=0pt
\item Objects are given by pairs $(M, \kappa)$, where $M$ is a finite projective $R$-module, and $\kappa\colon M \times_R M \to R$ is an $R$-bilinear map.
\item Morphisms $(M, \kappa) \to (M', \kappa')$ are $R$-module isomorphisms $\psi\colon M \to M'$ such that
\begin{gather*} \kappa'(\psi(m_1), \psi(m_2)) = \kappa(m_1,m_2) \qquad \text{for all} \quad m_1,m_2 \in M.\end{gather*}
\end{enumerate}
Given a homomorphism $f\colon R \to S$ of commutative rings, a descent datum for finite projective modules with bilinear form with respect to $f$ is a triple $(M, \kappa, \phi)$, where $(M, \phi)$ is a finite projective $S$-module with descent datum with respect to $f$, and $(\kappa, \phi, \phi, {\rm id}_S)$ is a bilinear map with descent datum with respect to $f$. An isomorphism of descent data from $(M, \kappa, \phi)$ to $(M', \kappa', \phi')$ is an $S$-module isomorphism $\psi\colon M \to M'$ such that $\kappa(m_1,m_2) = \kappa'(\psi(m_1),\psi(m_2))$ for all $m_1,m_2 \in M$, and the following diagram commutes:
\begin{gather*} \xymatrix{ M \otimes_{R} S \ar[r]^{\phi} \ar[d]_{\psi \otimes 1} & S \otimes_{R} M \ar[d]^{1 \otimes \psi} \\
M' \otimes_{R} S \ar[r]^{\phi'} & S \otimes_{R} M'. } \end{gather*}
A bilinear form $\kappa$ on a finite projective $R$-module is called non-singular if it induces an $R$-module isomorphism $M \to \Hom_R(M,R)$ under the canonical adjunction isomorphism $\Hom_R(M \otimes_{R} M, R) \to \Hom_R(M, \Hom_R(M,R))$.
\end{defn}

\begin{thm} \label{thm:descent-for-finite-projective-modules-with-bilinear-form}
Let $f\colon R \to S$ be a homomorphism of commutative rings. If $f$ is faithfully flat, then the groupoid of finite projective $R$-modules with bilinear form is equivalent to the groupoid of descent data for finite projective modules with bilinear form with respect to $f$. Furthermore, the subgroupoid of non-singular forms is preserved under this equivalence.
\end{thm}
\begin{proof}Given a descent datum $(M, \kappa, \phi)$ for finite projective modules with bilinear form with respect to $f$, by Theorem \ref{thm:descent-for-modules} and Proposition \ref{prop:finite-projective-property-descends}, there exists a finite projective $R$-module $\tilde{M}$ such that base change yields the original descent datum for modules, and by Lemma \ref{lem:descent-for-bilinear-maps}, there is a unique $R$-bilinear form $\tilde{\kappa}$ whose base-change to $S$ yields the descent datum $(\kappa, \phi, \phi, {\rm id}_S)$. By Theorem~\ref{thm:descent-for-modules}, isomorphisms $\psi$ of descent data are in natural bijection with isomorphisms $\tilde{\psi}$ of $R$-modules that preserve $\tilde{\kappa}$.

We now consider the condition that a form is non-singular. Given an $R$-module map $\bar{\kappa}\colon M \to \Hom_R(M,R)$ that corresponds to a bilinear form $\kappa$, base change along $f$ yields an $S$-module map $\bar{\kappa} \otimes 1\colon M \otimes_{R} S \to \Hom_R(M,R) \otimes_{R} S$. Because $M$ is finite projective, the canonical map $\mathrm{can}\colon \Hom_R(M,R) \otimes_{R} S \to \Hom_S(M \otimes_{R} S, S)$ is an isomorphism. We then have a commutative diagram
\begin{gather*} \xymatrix{ M \otimes_{R} S \ar[r]^-{\bar{\kappa} \otimes 1} \ar[rd]_-{\overline{\kappa\otimes 1}} & \Hom_R(M,R) \otimes_{R} S \ar[d]^{\mathrm{can}} \\ & \Hom_S(M \otimes_{R} S, S), } \end{gather*}
where $\overline{\kappa\otimes 1}$ is the $S$-module homomorphism that naturally corresponds to the $S$-valued bilinear form $\kappa \otimes 1$ on $M \otimes_{R} S$. This form is non-singular if and only if $\bar{\kappa} \otimes 1$ is an isomorphism. By faithfully flat descent, $\bar{\kappa} \otimes 1$ is an isomorphism if and only if $\bar{\kappa}$ is an isomorphism, i.e., if and only if $\kappa$ is non-singular.
\end{proof}

We consider a method of gluing that is well-suited to the objects we obtain in this paper.

\begin{defn}Suppose we are given a diagram $R_1 \to R_3 \leftarrow R_2$ of commutative rings. A~gluing datum for modules over this diagram is a triple $(M_1, M_2, f)$, where
\begin{enumerate}\itemsep=0pt
\item[1)] $M_1$ is an $R_1$-module,
\item[2)] $M_2$ is an $R_2$-module, and
\item[3)] $f\colon M_1 \otimes_{R_1} R_3 \simto M_2 \otimes_{R_2} R_3$ is an $R_3$-module isomorphism.
\end{enumerate}
If $P$ is a property of modules over a commutative ring, we say that a gluing datum for modules satisfies property $P$ if the component modules satisfy property $P$. A morphism of gluing data from $(M_1, M_2, f)$ to $(M'_1, M'_2, f')$ is a pair $(g_1\colon M_1 \to M'_1,\, g_2\colon M_2 \to M'_2)$ of module maps such that the following diagram commutes:
\begin{gather*} \xymatrix{ M_1 \otimes_{R_1} R_3 \ar[r]^-{f} \ar[d]_{g_1 \otimes {\rm id}} & M_2 \otimes_{R_2} R_3 \ar[d]^{g_2 \otimes {\rm id}} \\ M'_1 \otimes_{R_1} R_3 \ar[r]_-{f'} & M'_2 \otimes_{R_2} R_3. } \end{gather*}
\end{defn}

We first note that gluing of quasi-coherent sheaves is effective for Zariski open covers.

\begin{lem} \label{lem:zariski-gluing-data-for-modules}Let $R$ be a commutative ring, and let $a,b \in R$ be coprime elements, i.e., the ideal generated by $a$ and $b$ is all of $R$. Let $R_a$ and $R_b$ be the respective localizations, i.e., such that $\Spec R_a$ and $\Spec R_b$ form a Zariski open cover of $\Spec R$. Then base change induces an equivalence between the category of $R$-modules and the category of gluing data for modules over the diagram $R_a \to R_{ab} \leftarrow R_b$. This restricts to an equivalence between the category of finite projective $R$-modules and the category of finite projective gluing data, and a corresponding equivalence for groupoids of finite projective modules with non-singular bilinear form.
\end{lem}
\begin{proof}
The first claim amounts to faithfully flat descent for the Zariski open cover $R \to R_a \oplus R_b$. All of the components of a descent datum for this cover are given in the gluing datum, because we have the identifications $R_a \otimes_{R} R_a = R_a$ and $R_b \otimes_{R} R_b = R_b$. A detailed proof can be found at \cite[\href{https://stacks.math.columbia.edu/tag/00EQ}{Lemma 00EQ}]{Stacks}. The remaining claims follow from the corresponding descent results in the previous section.
\end{proof}

We wish to consider situations where the rings in a gluing datum are not necessarily Zariski localizations of $R$. In particular, we will consider the case where the diagram of rings has the form $R_1 \to R_1 \otimes_{R} R_2 \leftarrow R_2$, where $i_1\colon R \to R_1$ and $i_2\colon R \to R_2$ are faithfully flat.

\begin{defn}Suppose $i_1\colon R \to R_1$ and $i_2\colon R \to R_2$ are faithfully flat. A gluing datum for the diagram $R_1 \to R_1 \otimes_{R} R_2 \leftarrow R_2$ is strict if there is some $R_1 \otimes_{R} R_2$-module $M_3$ such that $M_1$ and $M_2$ are contained in $M_3$ as an $R_1 \otimes_{R} R$-submodule and an $R \otimes_{R} R_2$-submodule, such that extension of scalars gives the identifications $M_1 \otimes_{R} R_2 \to M_3$ and $R_1 \otimes_{R} M_2 \to M_3$.
\end{defn}

\begin{lem} \label{lem:strict-gluing-data}Suppose $i_1\colon R \to R_1$ and $i_2\colon R \to R_2$ are faithfully flat. Then, any gluing datum $(M_1, M_2, f)$ for the diagram $R_1 \to R_1 \otimes_{R} R_2 \leftarrow R_2$ is isomorphic to a strict gluing datum $\big(M_1, f^{-1}(R \otimes_{R} M_2), {\rm id}_{M_1 \otimes_{R} R_2}\big)$. In particular, there is an equivalence between the category of gluing data for the diagram $R_1 \to R_1 \otimes_{R} R_2 \leftarrow R_2$ and the full subcategory of strict gluing data.
\end{lem}
\begin{proof}Because $i_1$ is faithfully flat (in particular, universally injective for modules), $R \otimes_{R} R_2$ is a subring of $R_1 \otimes_{R} R_2$. Because $f^{-1}$ is an $R_1 \otimes_{R} R_2$-module isomorphism, $f^{-1}(R \otimes_{R} M_2)$ is an $R \otimes_{R} R_2$-submodule of $M_1 \otimes_{R} R_2$, and transporting the action of $R_2$ through the canonical isomorphism $R_2 \simto R \otimes_{R} R_2$, we obtain an $R_2$-module isomorphism $g_2\colon f^{-1}(R \otimes_{R} M_2) \to M_2$. It is straightforward to see that extending scalars along $i_1$ yields the map $f$, once we make the identification $R_1 \otimes_{R} f^{-1}(R \otimes_{R} M_2) = M_1 \otimes_{R} R_2$, by $r \otimes f^{-1}(1 \otimes x) = f^{-1}(r \otimes x)$. Thus, the pair $({\rm id}, g_2)\colon \big(M_1, f^{-1}(R \otimes_{R} M_2), {\rm id}_{M_1 \otimes_{R} R_2}\big) \to (M_1, M_2, f)$ induces a commutative diagram:
\begin{gather*} \xymatrix{ M_1 \otimes_{R} R_2 \ar[d]_{=} \ar[r]^-{{\rm id}} & M_1 \otimes_{R} R_2 \ar[d]^{f} \\
M_1 \otimes_{R} R_2 \ar[r]^{f} & R_1 \otimes_{R} M_2, } \end{gather*}
and is therefore an isomorphism of gluing data.

We therefore have a full subcategory that spans all isomorphism classes of objects, so it is equivalent to the category of gluing data.
\end{proof}

\begin{lem} \label{lem:faithfully-flat-gluing-data-for-flat-modules}Let $i_1\colon R \to R_1$ and $i_2\colon R \to R_2$ be faithfully flat homomorphisms of commutative rings. Then base change induces an equivalence between the category of $R$-modules $($with morphisms given by arbitrary $R$-module maps$)$ and the category of gluing data for modules over the diagram $R_1 \to R_1 \otimes_{R} R_2 \leftarrow R_2$. This restricts to an equivalence between the category of finite projective $R$-modules and the category of finite projective gluing data, and a corresponding equivalence for groupoids of finite projective modules with non-singular bilinear form.
\end{lem}
\begin{proof}By Lemma \ref{lem:strict-gluing-data}, it suffices to consider strict gluing data instead of all gluing data. Let~$F$ be the base change functor that takes $R$-modules to strict gluing data. We first define a~quasi-inverse functor $G$. Let $(M_1, M_2, {\rm id})$ be a strict gluing datum. Thus, we assume $M_1$ is an $R_1 \otimes_{R} R$-submodule of some $R_1 \otimes_{R} R_2$-module~$M_3$, and similarly for the $R \otimes_{R} R_2$-module~$M_2$. From $M_2$, we have a canonical descent datum on $M_3$, i.e., an $(R_1 \otimes_{R} R_2) \otimes_{R_2} (R_1 \otimes_{R} R_2)$-module isomorphism
\begin{gather*} \phi_2\colon \ M_3 \otimes_{R_2}(R_1 \otimes_{R} R_2) \to (R_1 \otimes_{R} R_2) \otimes_{R_2} M_3, \end{gather*}
such that $M_2$ is the set of elements $x \in M_3$ satisfying $\phi_2(x \otimes (1 \otimes 1)) = (1 \otimes 1) \otimes x$. Similarly, there is a descent datum
\begin{gather*} \phi_1\colon \ M_3 \otimes_{R_1}(R_1 \otimes_{R} R_2) \to (R_1 \otimes_{R} R_2) \otimes_{R_1} M_3, \end{gather*}
such that $M_1$ is the set of elements $x \in M_3$ satisfying $\phi_1(x \otimes (1 \otimes 1)) = (1 \otimes 1) \otimes x$. We set $G(M_1, M_2,{\rm id})$ to be the $R$-module that is the intersection of $M_1$ and $M_2$ in $M_3$, or equivalently, the set of elements $x \in M_3$ satisfying both $\phi_1(x \otimes (1 \otimes 1)) = (1 \otimes 1) \otimes x \in (R_1 \otimes_{R} R_2) \otimes_{R_1} M_3$ and $\phi_2(x \otimes (1 \otimes 1)) = (1 \otimes 1) \otimes x \in (R_1 \otimes_{R} R_2) \otimes_{R_2} M_3$. Given a morphism of strict gluing data, we obtain a homomorphism of $R$-modules on the intersections, so $G$ is a functor.

We first check that the composite $G \circ F$ is isomorphic to identity on $R$-modules. Because $i_1$ and $i_2$ are faithfully flat, they are universally injective for modules, so any $R$-module $M$ injects into both $M \otimes_{R} R_1$ and $M \otimes_{R} R_2$. The intersection in $M \otimes_{R} R_1 \otimes_{R} R_2$ is then $M \otimes_{R} R \otimes_{R} R \cong M$, and we have our isomorphism.

We define a natural transformation from $F \circ G$ to identity by sending any object $(M_1, M_2, {\rm id})$ to the pair $(g_1,g_2)$, where $g_1\colon R_1 \otimes_{R} G(M_1,M_2,{\rm id}) \to M_1$ and $g_2\colon R_2 \otimes_{R} G(M_1,M_2,{\rm id}) \to M_2$ are given by restricting the action maps $R_1 \otimes_{R} M_1 \to M_1$ and $R_2 \otimes_{R} M_2 \to M_2$. To show that $F$ is an equivalence, it remains to show that $g_1$ and $g_2$ are isomorphisms of $R_1$-modules and $R_2$-modules, respectively. The base change of $g_1$ along $i_2$ is given by the restriction $g_3\colon (R_1 \otimes_{R} R_2) \otimes_{R} G(M_1,M_2,{\rm id}) \to M_3$ of the action map on $M_3$, and since $i_2$ is faithfully flat, this map is injective (resp.~surjective) if and only if $g_1$ is also. The base change of $g_2$ along $i_1$ is also given by the restricted action map, so it suffices to show that at least one of these three maps ($g_1$, $g_2$, and $g_3$) is injective and at least one is surjective.

To prove injectivity, we consider the following diagram:
\begin{gather*} \xymatrix{ R_1 \otimes_{R} G(M_1, M_2,{\rm id}) \otimes_{R} R_2 \ar[r]^-{{\rm id} \otimes \iota \otimes {\rm id}} \ar[d]_{g_1 \otimes {\rm id}} & R_1 \otimes_{R} M_1 \otimes_{R} R_2 \ar[r]^-{\cong} \ar[d]^{f} & (R_1 \otimes_{R} R_2) \otimes_{R_2} (M_1 \otimes_{R} R_2) \ar[d]^{\phi_2^{-1}} \\
M_1 \otimes_{R} R_2 \ar[r]^-{{\rm id} \otimes i_1 \otimes {\rm id}} & M_1 \otimes_{R} R_1 \otimes_{R} R_2 \ar[r]^-{\cong} & (M_1 \otimes_{R} R_2) \otimes_{R_2} (R_1 \otimes_{R} R_2), } \end{gather*}
where $\iota\colon G(M_1, M_2,{\rm id}) \to M_1$ is the inclusion, and $f$ is the unique $R_1 \otimes_{R} R_1 \otimes_{R} R_2$-module isomorphism that makes the square on the right commute. The horizontal arrows on the left are injective homomorphisms, so to prove injectivity of $g_1$, it suffices to check commutativity of the left square. Along the lower left path, we have
\begin{gather*} ({\rm id} \otimes i_1 \otimes {\rm id}) \circ (g_1 \otimes {\rm id})(r_1 \otimes m \otimes r_2) = r_1m \otimes 1 \otimes r_2, \end{gather*}
which is then sent to $(r_1m \otimes 1) \otimes (1 \otimes r_2)$. Along the upper right path, we note that for any $m \in G(M_1, M_2,{\rm id})$, $f(1 \otimes m \otimes 1)$ is identified with $\phi_2^{-1}((1 \otimes 1) \otimes (m \otimes 1)) = (m \otimes 1) \otimes (1 \otimes 1)$, so is equal to $m \otimes 1 \otimes 1$. Thus, we have{\samepage
\begin{gather*}
f \circ ({\rm id} \otimes \iota \otimes {\rm id})(r_1 \otimes m \otimes r_2) = f(r_1 \otimes m \otimes r_2)= f((r_1 \otimes 1 \otimes r_2)(1 \otimes m \otimes 1)) \\
\qquad{} = (r_1 \otimes 1 \otimes r_2) f(1 \otimes m \otimes 1) = (r_1 \otimes 1 \otimes r_2) (m \otimes 1 \otimes 1)
 = r_1 m \otimes 1 \otimes r_2.
\end{gather*}
This proves $g_1 \otimes {\rm id} = g_3$ is injective.}

We reduce the question of surjectivity to the case that $R$ is a local ring as follows. By Claim~5 of~\cite[\href{https://stacks.math.columbia.edu/tag/00HN}{Lemma 00HN}]{Stacks}, a map $M \to M'$ of $R$-modules (for any commutative ring $R$) is surjective if and only if for each prime ideal $P$ of $R$, the localized map $M_P \to M'_P$ of $R_P$-modules is surjective. Each $R_P$ is a local ring, so we shall assume $R$ is a local ring for the remainder of this proof.

Let $x \in M_1$, so $x \otimes 1 \in M_1 \otimes_{R} R_2$, and write
\begin{gather*} \phi_2((x \otimes 1) \otimes (1 \otimes 1) = \sum_i (s_i \otimes t_i) \otimes (x_i \otimes 1) \in (R_1 \otimes_{R} R_2) \otimes_{R_2} (M_1 \otimes_{R} R_2). \end{gather*}
Note that we may move multipliers in $R_2$ across the middle tensor product, and this is why we may write $1$ in the rightmost factor. For each $x_i \in M_1$, we write
\begin{gather*} \phi_2((x_i \otimes 1) \otimes (1 \otimes 1)) = \sum_j (s_{ij} \otimes t_{ij}) \otimes (y_{ij} \otimes 1). \end{gather*}
Then, because $\phi_2$ satisfies the cocycle condition, we have
 \begin{align*} \sum_i (s_i \otimes t_i) \otimes (1 \otimes 1) \otimes (x_i \otimes 1) &= \sum_{i,j} (s_i \otimes t_i) \otimes (s_{ij} \otimes t_{ij}) \otimes (y_{ij} \otimes 1) \\
&= \sum_i (s_i \otimes t_i) \otimes \phi_2((x_i \otimes 1) \otimes (1 \otimes 1))
\end{align*}
as elements of $(R_1 \otimes_{R} R_2) \otimes_{R_2}(R_1 \otimes_{R} R_2) \otimes_{R_2}(R_1 \otimes_{R} R_2)$. We now use our assumption that~$R$ is a local ring, and in particular, the property that any finitely generated submodule of an $R$-flat module is $R$-free. This property is not explicitly stated in Proposition~3.G of~\cite{M80}, which asserts that all finitely generated $R$-flat modules are $R$-free when $R$ is local, but Matsumura's proof of the proposition yields this stronger result with no substantial change. For the case at hand, we take the $R$-free submodules of $R_1$ and $R_2$ generated by the finite sets $\{ s_i\}$ and $\{t_i\}$, respectively. By choosing bases and performing a suitable rearrangement, we may assume the elements $s_i$ and $t_i$ in the sum $\sum_i (s_i \otimes t_i) \otimes (x_i \otimes 1)$ are basis elements of the free submodules, and that the summands are $R$-linearly independent. Then, the cocycle condition implies
\begin{gather*} \phi_2((x_i \otimes 1) \otimes (1 \otimes 1)) = (1 \otimes 1) \otimes (x_i \otimes 1) \end{gather*}
for all $i$, and in particular, the elements $x_i$ lie in $G(M_1,M_2,{\rm id})$. Because $\phi_2$ is an $(R_1 \otimes_{R} R_2) \otimes_{R_2} (R_1 \otimes_{R} R_2)$-module isomorphism, we can then write
\begin{align*}
\phi_2^{-1}((s_i \otimes t_i) \otimes (x_i \otimes 1)) &= ((s_i \otimes t_i) \otimes (1 \otimes 1))\cdot \phi_2^{-1}((1 \otimes 1) \otimes (x_i \otimes 1)) \\
&= ((s_i \otimes t_i) \otimes (1 \otimes 1))\cdot ((x_i \otimes 1) \otimes (1 \otimes 1)) \\
&= (s_i x_i \otimes t_i) \otimes (1 \otimes 1).
\end{align*}
Summing over $i$, we find that $x \otimes 1 = \sum_i s_i x_i \otimes t_i$, so the image of $g_3$ contains $M_1$. Because~$M_1$ is an $R_1$-form of $M_3$, the restricted action map $(R_1 \otimes_{R} R_2) \otimes_{R} M_1 \to M_3$ is surjective. The compatibility between iterated actions and ring multiplication then implies $g_3$ is surjective.

Summing up, we have proved that $g_3$ is an isomorphism, so $g_1$ and $g_2$ are isomorphisms. We therefore have a natural isomorphism from $F \circ G$ to identity, so we conclude that~$F$ is an equivalence of categories.

The claims about finite projective modules and bilinear forms then follow from essentially the same argument as in Theorem~\ref{thm:descent-for-finite-projective-modules-with-bilinear-form}.
\end{proof}

\subsection{Vertex operator algebras over commutative rings}

Vertex algebras over commutative rings were first defined in \cite{B86}. Since all of the examples we consider will be conformal with strong finiteness conditions, we will mainly consider vertex operator algebras, roughly following the treatment in \cite{DR13}.

\begin{defn}
A vertex algebra over a commutative ring $R$ is an $R$-module $V$ equipped with a~distinguished vector $\unit \in V$ and a left multiplication map $Y\colon V \to (\End V)\big[\big[z,z^{-1}\big]\big]$, written $Y(a,z) = \sum\limits_{n \in \bZ} a_n z^{-n-1}$ satisfying the following conditions:
\begin{enumerate}\itemsep=0pt
\item For any $a,b \in V$, $a_n b = 0$ for $n$ sufficiently large. Equivalently, $Y$ defines a multiplication map $V \otimes_{R} V \to V((z))$.
\item $Y(\unit,z) = {\rm id}_V z^0$ and $Y(a,z)\unit \in a + zV[[z]]$.
\item For any $r,s,t \in \bZ$, and any $u,v,w \in V$,
\begin{gather*} \sum_{i \geq 0} \binom{r}{i} (u_{t+i} v)_{r+s-i} w = \sum_{i \geq 0} (-1)^i \binom{t}{i} \big(u_{r+t-i}(v_{s+i}w) - (-1)^t v_{s+t-i}(u_{r+i} w)\big).\end{gather*}
Equivalently, the Jacobi identity holds:
 \begin{gather*}
x^{-1} \delta\left(\frac{y-z}{x}\right) Y(a,y) Y(b,z) - x^{-1} \delta\left(\frac{z-y}{-x}\right) Y(b,z) Y(a,y) \\
\qquad{} = z^{-1} \delta\left(\frac{y-x}{z}\right) Y(Y(a,x)b,z),
\end{gather*}
where $\delta(z) = \sum\limits_{n \in \bZ} z^n$, and $\delta\left(\frac{y-z}{x}\right)$ is expanded as a formal power series with non-negative powers of $z$, i.e., as $\sum\limits_{n \in \bZ, m \in \bZ_{\geq 0}} (-1)^m \binom{n}{m} x^{-n}y^{n-m}z^m$.
\end{enumerate}
A vertex algebra homomorphism $(V, \unit_V, Y_V) \to (W, \unit_W, Y_W)$ is an $R$-module homomorphism $\phi\colon V \to W$ satisfying $\phi(\unit_V) = \unit_W$ and $\phi(u_n v) = \phi(u)_n \phi(v)$ for all $u,v \in V$ and all $n \in \bZ$.
\end{defn}

\begin{rem}
Mason notes in \cite{M17} that a vertex algebra over $R$ is the same thing as a vertex algebra $V$ over $\bZ$ equipped with a vertex algebra homomorphism $R \to V$, where $R$ is viewed as a commutative vertex algebra with product $Y(a,z)b = a_{-1}b z^0 = ab$.
\end{rem}

\begin{defn}
Let $f\colon R \to S$ be a homomorphism of commutative rings. A descent datum for vertex algebras with respect to $f$ is a pair $(V,\phi)$, where $V$ is a vertex algebra over $S$, and $\phi\colon V \otimes_{R} S \to S \otimes_{R} V$ is an isomorphism of vertex algebras over $S \otimes_{R} S$ such that the cocycle condition holds, i.e., the following diagram of isomorphisms of vertex algebras over $S \otimes_{R} S \otimes_{R} S$ commutes:
\begin{gather*} \xymatrix{ V \otimes_{R} S \otimes_{R} S \ar[rr]^{\phi_{01} = \phi \otimes 1} \ar[rrd]_{\phi_{02}} & & S \otimes_{R} V \otimes_{R} S \ar[d]^{\phi_{12} = 1 \otimes \phi} \\ & & S \otimes_{R} S \otimes_{R} V.} \end{gather*}
Here, the maps $\phi_{i,j}$ are defined as in Definition \ref{defn:descent-datum-for-modules}. A morphism of descent data from $(V, \phi)$ to $(V', \phi')$ is a homomorphism $\psi\colon V \to V'$ of vertex algebras over $S$ such that the following diagram commutes:
\begin{gather*} \xymatrix{ V \otimes_{R} S \ar[r]^{\phi} \ar[d]_{\psi \otimes 1} & S \otimes_{R} V \ar[d]^{1 \otimes \psi} \\
V' \otimes_{R} S \ar[r]^{\phi'} & S \otimes_{R} V'. } \end{gather*}
\end{defn}

\begin{prop} \label{prop:descent-for-vertex-algebras}Faithfully flat descent of vertex algebras is effective. That is, if we are given a faithfully flat ring homomorphism $R \to S$, then there is an equivalence between the category of vertex algebras $\tilde{V}$ over $R$ and the category of vertex algebras $V$ over $S$ equipped with descent data $\phi$, given in one way by $\tilde{V} \mapsto (\tilde{V} \otimes_{R} S, \phi((v \otimes s) \otimes s') = s \otimes (v \otimes s'))$ and the other way by $(V, \phi) \mapsto \tilde{V} = \{v \in V \,|\, 1 \otimes v = \phi(v \otimes 1)\}$.
\end{prop}
\begin{proof}\looseness=1 We briefly explain why this follows from the effectiveness of faithfully flat descent of modules. Essentially, faithfully flat descent for modules with additional structure holds as long as that structure is defined by homomorphisms of modules, with conditions given by commutativity of diagrams. Here the additional structure is given by a distinguished map $S \to V$ induced by the unit vector, and the ``$z^{-n-1}$-coefficient maps'' $\cdot_n\colon V \otimes_S V \to V$. The vertex algebra axioms, such as the Jacobi identity, may be interpreted as equality of certain composites of such maps.

Now, suppose we are given a descent datum $(V, \phi)$ for vertex algebras with respect to $f$. The isomorphism $\phi$ induces a descent datum for the underlying $S$-module, so faithfully flat descent for modules, given as Theorem \ref{thm:descent-for-modules}, yields an $R$-module $\tilde{V}$ and a distinguished $S$-module isomorphism $\psi\colon \tilde{V} \otimes_{R} S \cong V$.

For each integer $n$, we have the ``$z^{-n-1}$-coefficient map'' $\cdot_n\colon V \otimes_S V \to V$, which we may now rewrite as $\cdot_n\colon \big(\tilde{V} \otimes_{R} \tilde{V}\big) \otimes_{R} S \to \tilde{V} \otimes_{R} S$ using $\psi$. To show that this is the base change of a map $\tilde{\cdot}_n\colon \tilde{V} \otimes_{R} \tilde{V} \to \tilde{V}$ along $f$, it is necessary and sufficient to show that $\cdot_n$ is a morphism of descent data for modules with respect to $f$.

By our assumption that $\phi$ is a descent datum for vertex algebras, the following diagram commutes:
\begin{gather*} \xymatrix{ \big(\tilde{V} \otimes_{R} \tilde{V}\big) \otimes_{R} S \otimes_{R} S \ar[r]^{\phi \otimes \phi} \ar[d]_{\cdot_n \otimes 1} & S \otimes_{R} \big(\tilde{V} \otimes_{R} \tilde{V}\big) \otimes_{R} S \ar[d]^{1 \otimes \cdot_n} \\
\tilde{V} \otimes_{R} S \otimes_{R} S \ar[r]^{\phi} & S \otimes_{R} \tilde{V} \otimes_{R} S,} \end{gather*}
so we have a map of descent data. Here, the map $\phi \otimes \phi$ along the top uses an identification $\big(\tilde{V} \otimes_{R} \tilde{V}\big) \otimes_{R} S \otimes_{R} S \cong (V \otimes_{R} S) \otimes_S (V \otimes_{R} S)$ induced by $\psi$.

Essentially the same argument shows that the unit map $S \to V$ that takes~$1$ to $\unit$ is the base change of an $R$-module map $R \to \tilde{V}$. The remaining vertex algebra conditions then follow from the fact that base-change along $f$ yields a formulas that hold for~$V$.
\end{proof}

Given a vertex algebra $V$ over $R$ and a vertex algebra $W$ over $S$, we say that $V$ is an $R$-form of $W$ with respect to a ring homomorphism $\phi\colon R \to S$ if $V \otimes_{R,\phi} S$ is isomorphic to~$W$ as a vertex algebra over $S$. If $\phi$ is implicitly fixed, we will simply say that $V$ is an $R$-form of $W$. Thus, the previous proposition amounts to the claim that the construction of an $R$-form with respect to a~faithfully flat map is equivalent to the construction of a descent datum.

\begin{defn}
Let $c\in R$. A conformal vertex algebra over $R$ with half central charge $c$ is a vertex algebra over $R$ equipped with a $\bZ$-grading $V = \bigoplus_{n \in \bZ} V_n$ and a distinguished vector $\omega \in V_2$ such that:
\begin{enumerate}\itemsep=0pt
\item The coefficients of $Y(\omega,z) = \sum\limits_{k \in \bZ} L_k z^{-k-2}$ satisfy the Virasoro relations at half central charge $c$, i.e.,
\begin{gather*} [L_m, L_n] = (m-n)L_{m+n} + c \binom{m+1}{3} \delta_{m+n,0} {\rm id}_V. \end{gather*}
\item The product structure is homogeneous: $u_k v \in V_{m+n-k-1}$ for $u \in V_m$, $v \in V_n$.
\item $L_0 v = nv$ for all $v \in V_n$.
\item $L_{-1}v$ is the $z^1$-term in $Y(v,z)\unit$.
\end{enumerate}
A vertex operator algebra over $R$ with half central charge $c$ is a conformal vertex algebra such that all $V_n$ are finite projective $R$-modules, and $V_n = 0$ for $n \ll 0$.
\end{defn}

\begin{rem}Our definition of vertex operator algebra over $R$ differs slightly from the defi\-nition in \cite{DR13}. We remove the condition that the base ring $R$ be an integral domain in which~2 is invertible, and we replace the condition that each~$V_n$ be a finite free $R$-module with the condition that each~$V_n$ be a finite projective $R$-module. The first change is necessary for us to consider the case $R = \bZ$, and the second change is necessary in order for faithfully flat descent of vertex operator algebras to be effective in general (see Proposition~\ref{prop:faithfully-flat-descent-for-voas}). That is, free modules don't necessarily glue to form free modules, and standard examples come from locally trivial but nontrivial vector bundles. This distinction doesn't affect the strength of our final result, since $\bZ$ is a principal ideal domain.
\end{rem}

\begin{rem}It may be reasonable someday to change the definition of ``vertex operator algebra over $R$'' to require an action of the smooth integral form $U^+(Vir)$ of the universal enveloping algebra of Virasoro constructed in Section~5 of~\cite{B99}. The self-dual integral form of~$V^\natural$ constructed in this paper admits such an action by virtue of the same being true of $(V_\Lambda)_{\bZ}$. Furthermore, the proof of Modular Moonshine for primes greater than~13 in~\cite{B98} assumes the existence of such an action. However, this condition is at the moment too difficult to verify in general to be particularly useful.
\end{rem}

\begin{defn}A vertex algebra homomorphism from $V$ to $W$ is an $R$-linear map $\phi\colon V \to W$ satisfying $\phi(\unit_V) = \unit_W$ and $\phi(Y_V(u,z)v) = Y_W(\phi(u),z)\phi(v)$ for all $u,v \in V$. A vertex operator algebra homomorphism from $V$ to $W$ is a vertex algebra homomorphism $\phi\colon V \to W$ that preserves the $\bZ$-grading, and sends $\omega_V$ to $\omega_W$.
\end{defn}

\begin{rem}The condition that $\phi(\omega_V) = \omega_W$ is sometimes reserved for a special class of homomorphisms, using terms like ``strong'' or ``strict'' or ``conformal'' homomorphism in the literature. This is because it is often fruitful to consider vertex algebra maps between vertex operator algebras of different central charge. We will not need to consider such maps in this paper.
\end{rem}

\begin{prop} \label{prop:faithfully-flat-descent-for-voas}Faithfully flat descent of vertex operator algebras is effective. That is, if we are given a faithfully flat ring homomorphism $R \to S$, then there is an equivalence between the category of vertex operator algebras $\bar{V}$ over $R$ and the category of vertex operator algebras~$V$ over~$S$ equipped with descent data $\phi\colon V \otimes_{R} S \to S \otimes_{R} V$, given in one way by $\bar{V} \mapsto (\bar{V} \otimes_{R} S,$ $\phi((v \otimes s) \otimes s') = s \otimes (v \otimes s'))$ and the other way by $(V, \phi) \mapsto \bar{V} = \{v \in V | 1 \otimes v = \phi(v \otimes 1)\}$.
\end{prop}
\begin{proof}By Proposition \ref{prop:descent-for-vertex-algebras}, faithfully flat descent for vertex algebras is effective. It then suffices to show that the conformal vector $\omega$ descends, and that the finite projective property of modules descends. The first assertion is clear from the fact that~$\phi$ sends conformal vectors to conformal vectors. The second assertion is given in \cite[\href{https://stacks.math.columbia.edu/tag/058S}{Proposition 058S}]{Stacks} (or~\cite[Proposition~2.5.2]{EGA4v2}).
\end{proof}

\begin{rem}This result allows us to consider vertex algebras and vertex operator algebras over a broad class of geometric objects, including all schemes and all algebraic stacks.
\end{rem}

\begin{defn}A M\"obius structure on a vertex algebra over $R$ is an integer grading, together with an action of the Kostant integral form $U(\fsl_2)_{\bZ} = \big\langle L_{-1}^{(n)}, \binom{L_0}{n}, L_1^{(n)} \big\rangle_{n \geq 1}$ on the underlying $R$-module, such that the following conditions hold:
\begin{enumerate}\itemsep=0pt
\item $Y(u,z)\unit = \sum\limits_{n \geq 0} L_{-1}^{(n)}u z^n$ for all $u \in V$.
\item If $u$ has weight $k$, then $\binom{L_0}{n}u = \binom{k}{n}u$.
\item The subalgebra generated by $\{ L_1^{(n)} \}_{n \geq 1}$ acts locally nilpotently on $V$.
\item For all $i \in \{-1,0,1\}$, $L_i Y(u,z)v - Y(u,z)(L_i v) = \sum\limits_{j =0}^{i+1} \binom{i+1}{j} z^j Y(L_{i-j} u,v)$.
\end{enumerate}
We note that in the embedding of $U(\fsl_2)_{\bZ}$ in $U(\fsl_2)$, we have $L_i^{(n)} = \frac{L_i^n}{n!}$.
\end{defn}

\begin{defn}Let $V$ be a vertex algebra with M\"obius structure, and let $u \in V$ have weight~$k$. For each integer $n$, we define the operator $(u_n)^* = (-1)^k \sum\limits_{i \geq 0} \big(L_1^{(i)} u\big)_{2k-i-n-2}$. An invariant bilinear form on a vertex algebra $V$ with M\"obius structure is a symmetric bilinear form $(\,,\,)\colon V \times V \to R$ such that
\begin{gather*} (u_n v,w) = (v, (u_n)^* w) \end{gather*}
for all homogeneous vectors $u$ and all vectors $v$ and $w$. That is, $(u_n)^*$ is the adjoint of the operator~$u_n$. We say that a vertex operator algebra is self-dual with respect to an invariant bilinear form if each graded piece $V_n$ is self-dual as a finite projective $R$-module.
\end{defn}

\begin{rem}If $R$ contains $\bZ$ as a subring, then we may write the inner product formula equivalently as $(Y(u,z)v,w) = \big(v, Y\big({\rm e}^{zL_1}\big({-}z^{-2}\big)^{L_0}u, z^{-1}\big)w\big)$. Otherwise, we must replace $L_0$ with an integer indicator of weight.
\end{rem}

\begin{prop} \label{prop:faithfully-flat-descent-for-Mobius-forms-and-self-duality}Faithfully flat descent is effective for both M\"obius structure and invariant bilinear forms, on both vertex algebras and vertex operator algebras. Furthermore, the self-dual property for vertex operator algebras with invariant bilinear forms descends effectively.
\end{prop}
\begin{proof}For M\"obius structure, the generating operators descend as module maps, and automatically satisfy the defining relations. For invariant bilinear forms, it suffices to use descent of module maps $V \otimes_{R} V \to R$ as in Lemma~\ref{lem:descent-for-bilinear-maps}. For the self-dual property, we may reduce to the case of self-duality of the finite projective modules~$V_n$, and the isomorphism property of the map $V_n \to \Hom_R(V_n,R)$ descends effectively as in Theorem~\ref{thm:descent-for-finite-projective-modules-with-bilinear-form}.
\end{proof}

We use this to describe gluing explicitly.

\begin{defn}Suppose we are given the following diagram of commutative ring homomorphisms: $R_1 \to R_3 \leftarrow R_2$. A gluing datum for vertex operator algebras over this diagram is a triple $\big(V^1, V^2, f\big)$, where
\begin{enumerate}\itemsep=0pt
\item $V^1$, $V^2$ are vertex operator algebras over $R_1$ and $R_2$, respectively.
\item $f\colon V^1 \otimes_{R_1} R_3 \to V^2 \otimes_{R_1} R_3$ is an isomorphism of vertex operator algebras over $R_3$.
\end{enumerate}
A gluing datum is M\"obius and self-dual if the corresponding vertex operator algebras are M\"obius and self-dual, and the map $f$ preserves the M\"obius structure and bilinear form. A morphism of gluing data from $\big(V^1, V^2, f\big)$ to $\big(V^{\prime,1}, V^{\prime,2}, f'\big)$ is a pair $\big(g_1\colon V^1 \to V^{\prime,1}, g_2\colon V^2 \to V^{\prime,2}\big)$ of vertex algebra homomorphisms such that the following diagram commutes:
\begin{gather*} \xymatrix{ V^1 \otimes_{R_1} R_3 \ar[r]^-{f} \ar[d]_{g_1 \otimes {\rm id}} & V^2 \otimes_{R_2} R_3 \ar[d]^{g_2 \otimes {\rm id}} \\ V^{\prime,1} \otimes_{R_1} R_3 \ar[r]_-{f'} & V^{\prime,2} \otimes_{R_2} R_3. } \end{gather*}
\end{defn}

\begin{lem} \label{lem:zariski-gluing-for-voas}Let $R$ be a commutative ring, and let $a,b \in R$ be coprime elements, i.e., the ideal generated by~$a$ and~$b$ is all of~$R$. Let~$R_a$ and~$R_b$ be the respective localizations, i.e., such that $\Spec R_a$ and $\Spec R_b$ form a Zariski open cover of $\Spec R$. Then base change induces an equivalence between the category of vertex operator algebras over $R$ and the category of gluing data for vertex operator algebras over the diagram $R_a \to R_{ab} \leftarrow R_b$. This also yields equivalences between the corresponding categories of self-dual M\"obius objects.
\end{lem}
\begin{proof}We obtain a quasi-inverse functor by applying Lemma~\ref{lem:zariski-gluing-data-for-modules} to obtain the underlying modules and structure maps.
\end{proof}

\begin{lem} \label{lem:faithfully-flat-gluing-data-for-voas}Let $R \to R_1$ and $R \to R_2$ be faithfully flat homomorphisms of commutative rings. Then tensor product induces an equivalence between the category of vertex operator algebras over $R$ and the category of gluing data for vertex operator algebras over the diagram $R_1 \to R_1 \otimes_{R} R_2 \leftarrow R_2$. This also yields equivalences between the corresponding categories of self-dual M\"obius objects.
\end{lem}
\begin{proof}We first note that faithful flatness implies the tensor product functor preserves the finite projective property of graded pieces, so it takes vertex operator algebras over $R$ to gluing data for vertex operator algebras over the diagram $R_1 \to R_1 \otimes_{R} R_2 \leftarrow R_2$. Similarly, tensor pro\-duct preserves the self-dual M\"obius structure. We obtain the quasi-inverse functor by applying Lemma~\ref{lem:faithfully-flat-gluing-data-for-flat-modules} to obtain the underlying modules and the structure maps. It remains to check that the conformal vertex algebra over~$R$ that we obtain is in fact a vertex operator algebra over $R$, and this follows from Proposition~\ref{prop:faithfully-flat-descent-for-voas}. For self-dual M\"obius structure, the claim follows from Proposition~\ref{prop:faithfully-flat-descent-for-Mobius-forms-and-self-duality}.
\end{proof}

\begin{lem} \label{lem:unique-voa-from-descent}Let $R \to R_1$ and $R \to R_2$ be homomorphisms of commutative rings of the form given in either Lemma~{\rm \ref{lem:zariski-gluing-for-voas}} or Lemma~{\rm \ref{lem:faithfully-flat-gluing-data-for-voas}}. Let $V^1$ and $V^2$ be vertex operator algebras over~$R_1$ and~$R_2$, respectively, such that $V^1 \otimes_{R} R_2$ and $V^2 \otimes_{R} R_1$ are isomorphic as vertex operator algebras over $R_1 \otimes_{R} R_2$. Any choice of gluing datum induces injective group homomorphisms $\Aut_{R_1} V^1 \to \Aut_{R_1 \otimes_{R} R_2} \big(V^1 \otimes_{R} R_2\big)$ and $\Aut_{R_2} V^2 \to \Aut_{R_1 \otimes_{R} R_2} \big(V^1 \otimes_{R} R_2\big)$, and if the corresponding double coset space is a singleton, then there is a unique isomorphism type of vertex operator algebra~$V$ over~$R$ such that $V^1 \cong V \otimes_{R} R_1$ and $V^2 \cong V \otimes_{R} R_2$. In particular, if the inclusions are isomorphisms, then $V$ is unique, and the natural inclusions of $\Aut_R V$ into $\Aut_{R_1} V^1$ and $\Aut_{R_2} V^2$ are isomorphisms. The corresponding statement also holds for M\"obius self-dual vertex operator algebras.
\end{lem}
\begin{proof}There is a natural action of $\Aut_{R_1} V^1 \times \Aut_{R_2} V^2$ on the set of gluing data by isomorphisms, given by
\begin{gather*} f \mapsto (g_2 \otimes {\rm id}) \circ f \circ (g_1 \otimes {\rm id})^{-1} = f', \end{gather*}
because such isomorphisms are defined by the condition that the following diagram commutes:
\begin{gather*} \xymatrix{ V^1 \otimes_{R_1} R_3 \ar[r]^-{f} \ar[d]_{g_1 \otimes {\rm id}} & V^2 \otimes_{R_2} R_3 \ar[d]^{g_2 \otimes {\rm id}} \\ V^1 \otimes_{R_1} R_3 \ar[r]_-{f'} & V^2 \otimes_{R_2} R_3 . } \end{gather*}
Because isomorphisms of gluing data correspond bijectively to isomorphisms of vertex operator algebras over $R$, we get a unique vertex operator algebra over $R$ if and only if the set of orbits is a singleton. For any two choices of gluing data $f$ and $f'$, there exists a unique $h \in \Aut_{R_1 \otimes_{R} R_2} \big(V^1 \otimes_{R} R_2\big)$ such that $h \circ f = f'$. This establishes a non-canonical bijection between gluing data and elements of $\Aut_{R_1 \otimes_{R} R_2} \big(V^1 \otimes_{R} R_2\big)$. Fixing some $f$ as a base point induces inclusion maps $\phi_1\colon \Aut_{R_1} V^1 \to \Aut_{R_1 \otimes_{R} R_2} \big(V^1 \otimes_{R} R_2\big)$ and $\phi_2\colon \Aut_{R_2} V^2 \to \Aut_{R_1 \otimes_{R} R_2} \big(V^1 \otimes_{R} R_2\big)$, and identifies the action with $h \mapsto \phi_2(g_2) h \phi_1(g_1)^{-1}$. This identifies orbits with double cosets in $\phi_2\big(\Aut_{R_2} V^2\big) \backslash \Aut_{R_1 \otimes_{R} R_2}\big(V^1 \otimes_{R} R_2\big) / \phi_1\big(\Aut_{R_1} V^1\big)$.

For the special case that $\phi_1$ and $\phi_2$ are isomorphisms, there is clearly a single double coset, hence a single isomorphism type of vertex operator algebra $V$ over $R$. Furthermore, each $R_1 \otimes_{R} R_2$-automorphism of $\big(V^1 \otimes_{R} R_2\big)$ descends to an $R$-automorphism of $V$, by effective descent of morphisms given by the equivalences of categories in Lemmas~\ref{lem:zariski-gluing-for-voas} and~\ref{lem:faithfully-flat-gluing-data-for-voas}
\end{proof}

The following fact is rather elementary, but we will use it a lot.

\begin{lem} \label{lem:decomposition-is-still-self-dual}Let $V$ be a vertex operator algebra over a ring $R$ with M\"obius structure and an invariant bilinear form, and suppose $V$ is self-dual. Let $g$ be an automorphism of order $n \in \bZ_{\geq 1}$ that preserves the bilinear form. If~$n$ is invertible in $R$ and $\mu_n(R) \cong \bZ/n\bZ$, then~$V$ is isomorphic to the direct sum of its $g$-eigenspaces $\bigoplus_{\zeta \in \mu_n(R)} V^{g = \zeta}$, as an $R$-module with bilinear form. Furthermore, each eigenspace is a module for the fixed-point vertex operator subalgebra~$V^{g=1}$, and the bilinear form induces a homogeneous perfect pairing between $V^{g = \zeta}$ and~$V^{g = 1/\zeta}$ for all $\zeta \in \mu_n(R)$.
\end{lem}
\begin{proof}The hypotheses on $R$ imply any action of $\bZ/n\bZ$ on an $R$-module gives a unique decomposition into eigenspaces for $n$-th roots of unity. For $V$, the decomposition yields modules for the fixed-point vertex operator subalgebra. Self-duality of $V$ implies the eigenspaces with reciprocal eigenvalues are paired.
\end{proof}

\subsection{Abelian intertwining algebras over subrings of complex numbers}

We suspect that with enough work, one may extend the definition of abelian intertwining algebra, introduced in \cite{DL93}, to allow definition over arbitrary commutative rings with suitable divisibility properties and equipped with distinguished roots of unity. However, the results in this paper only require us to consider base rings $R$ that lie inside $\bC$, and the formalism becomes somewhat more manageable in this case.

Recall from \cite{M52} that if we are given abelian groups $(A,+)$ and $(B,\cdot)$, an abelian 3-cocycle on $A$ with coefficients in $B$ is a pair $(F\colon A \times A \times A \to B, \Omega\colon A \times A \to B)$ of functions satisfying the Eilenberg--MacLane cocycle conditions:
\begin{enumerate}\itemsep=0pt
\item[1)] $F(i,j,k) \cdot F(i,j+k,\ell) \cdot F(j,k,\ell) = F(i+j,k,\ell) \cdot F(i,j,k+\ell)$ for all $i,j,k,\ell \in A$,
\item[2)] $F(i,j,k)^{-1} \cdot \Omega(i,j+k) \cdot F(j,k,i)^{-1} = \Omega(i,j) \cdot F(j,i,k)^{-1} \cdot \Omega(i,k)$,
\item[3)] $F(i,j,k) \cdot \Omega(i+j,k) \cdot F(k,i,j) = \Omega(j,k) \cdot F(i,k,j) \cdot \Omega(i,k)$,
\end{enumerate}
and the restriction of $\Omega$ to the diagonal $i \mapsto \Omega(i,i)$ gives a bijection from abelian cohomology classes to $B$-valued quadratic functions on~$A$. We say that an abelian 3-cocycle is normalized if $F(i,j,0) = F(i,0,k) = F(0,j,k) = 1$ and $\Omega(i,0) = \Omega(0,j) = 1$ for all $i,j,k \in A$, and we note that any abelian cohomology class admits a normalized representative.

\begin{defn}
Let $N \in \bZ_{\geq 1}$, let $R$ be a commutative subring of $\bC$ containing $\frac{1}{N}$ and $e(1/2N)$. Let $A$ be an abelian group, let $(F,\Omega)$ be a normalized abelian 3-cocycle on $A$ with coefficients in $R^\times$, and assume $\Omega(a,a)^N = 1$ for all $a \in A$. We then let $q_\Omega$ be the unique $\frac{1}{N}\bZ/\bZ$-valued quadratic form on $A$ such that $\Omega(a,a) = e(q_\Omega(a))$, and let $b_\Omega\colon A \times A \to \frac{1}{N}\bZ$ be any fixed function that reduces to the bilinear form attached to $q_\Omega$, i.e., $b_\Omega(a,b) \equiv q_\Omega(a+b) - q_\Omega(a) - q_\Omega(b) \mod \bZ$. Then an abelian intertwining algebra over $R$ of level $N$ and half central charge $c \in R$ associated to the data $(A, F, \Omega)$ is an $R$-module $V$ equipped with
\begin{enumerate}\itemsep=0pt
\item[1)] a $\frac{1}{N}\bZ \times A$-grading $V = \bigoplus_{n \in \frac{1}{N} \bZ} V_n = \bigoplus_{i \in A} V^i = \bigoplus_{n \in \frac{1}{N} \bZ, i \in A} V_n^i$,
\item[2)] a left-multiplication operation $Y\colon V \to (\End V)\big[\big[z^{1/N},z^{-1/N}\big]\big]$ that we expand as $Y(a,z)$ $= \sum\limits_{n \in \frac{1}{N}\bZ} a_n z^{-n-1}$, and
\item[3)] distinguished vectors $\unit \in V_0^0$ and $\omega \in V_2^0$,
\end{enumerate}
satisfying the following conditions for any $i,j,k \in A$, $a \in V^i$, $b \in V^j$, $u \in V^k$, and $n \in \frac{1}{N}\bZ$:
\begin{enumerate}\itemsep=0pt
\item[1)] $a_n b \in V^{i+j}$,
\item[2)] $a_n b = 0$ for $n$ sufficiently large,
\item[3)] $Y(\unit, z)a = a$,
\item[4)] $Y(a,z)\unit \in a + zV[[z]]$,
\item[5)] the coefficients of $Y(\omega,z) = \sum\limits_{k \in \bZ} L_k z^{-k-2}$ satisfy the Virasoro relations at half central charge $c$,
\item[6)] $L_0 v = nv$ for all $v \in V_n$,
\item[7)] $\frac{{\rm d}}{{\rm d}z} Y(a,z) = Y(L_{-1}a,z)$ for all $a \in V$,
\item[8)] $Y(a,z)b = \sum\limits_{k \in b_\Omega(i,j) + \bZ} a_k b z^{-k-1}$,
\item[9)] the Jacobi identity:
 \begin{gather*}
 x^{-1} \left(\frac{y-z}{x} \right)^{b_\Omega(i,j)} \delta\left(\frac{y-z}{x}\right) Y(a,y) Y(b,z) u \\
\qquad{}- B(i,j,k) x^{-1} \left(\frac{z-y}{-x} \right)^{b_\Omega(i,j)} \delta\left(\frac{z-y}{-x}\right) Y(b,z) Y(a,y) u \\
 \qquad\quad{} = F(i,j,k) z^{-1} \delta\left(\frac{y-x}{z}\right) Y(Y(a,x)b,z) \left(\frac{y-x}{z} \right)^{-b_\Omega(i,k)} u,
\end{gather*}
where $B(i,j,k) = \frac{\Omega(i,j)F(i,j,k)}{F(j,i,k)}$, and $(-x)^{k/N}$ is interpreted as $e(k/2N)x^{k/N}$.
\end{enumerate}
We say an abelian intertwining algebra $V$ is well-graded if each piece $V_n^i$ is a finite projective $R$-module. A M\"obius structure on $V$ is an $A$-homogeneous action of $U(\fsl_2)_{\bZ} = \big\langle L_{-1}^{(n)}, \binom{L_0}{n}, L_1^{(n)} \big\rangle_{n \geq 1}$ on the underlying $R$-module, such that the following conditions hold:
\begin{enumerate}\itemsep=0pt
\item[1)] $Y(u,z)\unit = \sum\limits_{n \geq 0} L_{-1}^{(n)}u z^n$ for all $u \in V$,
\item[2)] if $u$ has weight $k$, then $\binom{L_0}{n}u = \binom{k}{n}u$,
\item[3)] the subalgebra generated by $\{ L_1^{(n)} \}_{n \geq 1}$ acts locally nilpotently on $V$,
\item[4)] for all $i \in \{-1,0,1\}$, $L_i Y(u,z)v - Y(u,z)(L_i v) = \sum\limits_{j =0}^{i+1} \binom{i+1}{j} z^j Y(L_{i-j} u,v)$.
\end{enumerate}
\end{defn}

We will need the following skew-symmetry property, which can be (more or less) found in the proof of Proposition~2.6 of~\cite{GL00}.

\begin{lem} \label{lem:skew-symmetry-for-AIA}If $V$ is an abelian intertwining operator algebra, then $V$ satisfies the following skew-symmetry property:
\begin{gather*} Y(a,y)b = \Omega(i,j){\rm e}^{yL_{-1}}Y\big(b,{\rm e}^{\pi {\rm i}}y\big)a \end{gather*}
for $a \in V^i$ and $b \in V^j$. In particular, for any $n \in b_\Omega(i,j) + \bZ$, $a_n b$ is an $R$-linear combination of terms of the form $b'_k a'$ for $b' \in V^j$ and $a' \in V^i$.
\end{lem}
\begin{proof}We consider the Jacobi identity with vectors $a$, $b$, $\unit$. Multiplying both sides by $x^{b_\Omega(i,j)+N}$ for some $N \in \bZ_{>0}$ and taking $\Res_x$ yields:
\begin{gather*} (y-z)^{b_\Omega(i,j)+N}Y(a,y) Y(b,z) \unit - B(i,j,0) {\rm e}^{(-\pi {\rm i})(b_\Omega(i,j)+N)}(z-y)^{b_\Omega(i,j)+N} Y(b,z) Y(a,y) \unit \end{gather*}
on the left side, and some power series with powers of $x$ uniformly bounded below times $Y(Y(a,x)b,z)$ on the right. When $N$ is sufficiently large, the right side therefore vanishes. Because our cocycle is normalized, we have the following equality in $V^{i+j}\big[\big[y^{\pm 1},z^{\pm 1}\big]\big]$:
\begin{gather*} (y-z)^{b_\Omega(i,j)+N}Y(a,y) {\rm e}^{zL_{-1}}b = \Omega(i,j) {\rm e}^{(-\pi {\rm i})(b_\Omega(i,j)+N)}(z-y)^{b_\Omega(i,j)+N} Y(b,z) {\rm e}^{yL_{-1}}a. \end{gather*}
This is essentially a version of the ``operator-valued rational function'' statement in Remark~12.31 of~\cite{DL93}. The left side has no negative powers of~$z$, and the right side has no negative powers of~$y$, so this is an equality in $V^{i+j}[[y,z]]$. By the exponential translation formula, we find that
\begin{gather*} (y-z)^{b_\Omega(i,j)+N}Y(a,y) {\rm e}^{zL_{-1}}b = \Omega(i,j) {\rm e}^{(-\pi {\rm i})(b_\Omega(i,j)+N)}(z-y)^{b_\Omega(i,j)+N} {\rm e}^{yL_{-1}} Y(b,z-y) a. \end{gather*}
Because this is an equality in $V^{i+j}[[y,z]]$, we may set $z=0$ to get the answer we want, after substituting ${\rm e}^{\pi {\rm i}}y + z$ for $z-y$.
\end{proof}

We also need an associativity result, adapted from Lemma~3.12 of~\cite{L99}. Li credits~\cite{DLM96} with this result, and I suspect he is referring to an argument leading up to Remark~3.3 in that paper.

\begin{lem} \label{lem:shorten-words}
Let $a \in V^i$, $b \in V^j$, $u \in V^k$ be homogeneous elements of an abelian intertwining algebra. Let $r \in b_\Omega(i,j+k)+\bZ$ and $s \in b_\Omega(j,k) + \bZ$. Let $l \in b_\Omega(i,k) + \bZ_{\geq 0}$ be such that $z^l Y(a,z) u \in V[[z]]$, and let $m \in \bZ_{\geq 0}$ be such that $z^{m+s} Y(b,z)u \in V[[z]]$. Then
\begin{gather*} a_r b_s u = F(i,j,k) \sum_{t=0}^m \sum_{n \in \bZ_{\geq 0}} \binom{r-l}{t} \binom{l}{n} (a_{r-l-t+n}b)_{s+l+t-n} u.\end{gather*}
In particular $($since this is always a finite sum$)$, $a_r b_s u$ is an $R$-linear combination of elements of the form $(a_p b)_q u$.
\end{lem}
\begin{proof}By Proposition 2.5 of \cite{L99a} (which is essentially the part of Remark~12.31 in~\cite{DL93} that applies to ``generalized rationality of iterates''), the following associativity rule holds:
\begin{gather*} (x + z)^l Y(a, x + z)Y(b,z)u = F(i,j,k)(z + x)^l Y(Y(a, x)b, z)u \end{gather*}
as elements of $z^{-m-s}x^{-b_\Omega(i,j)-N}V^{i+j+k}[[z,x]]$ for some $N \in \bZ$.
To extract $a_r b_s u$ from the left side, we rewrite:
\begin{align*}
a_r b_s u &= \Res_{y} \Res_{z} y^r z^s Y(a,y) Y(b,z)u \\
&= \Res_{x} \Res_{y} \Res_{z} x^{-1} \delta\left(\frac{y-z}{x} \right) y^r z^s Y(a,y) Y(b,z)u \\
&= \Res_{x} \Res_{y} \Res_{z} y^{-1} \delta\left(\frac{x+z}{y} \right) y^r z^s Y(a,y) Y(b,z)u \\
&= \Res_{x} \Res_{y} \Res_{z} y^{-1} \delta\left(\frac{x+z}{y} \right) (x+z)^r z^s Y(a,x + z) Y(b,z)u \\
&= \Res_{x} \Res_{z} (x+z)^r z^s Y(a,x + z) Y(b,z)u.
\end{align*}
Combining this with the associativity rule, we find that
 \begin{gather*}\begin{split}&
a_r b_s u = \Res_{x} \Res_{z} (x+z)^{r-l} z^s \big[(x+z)^l Y(a,x + z) Y(b,z)u \big] \\
& \hphantom{a_r b_s u }{}= F(i,j,k) \Res_{x} \Res_{z} (x+z)^{r-l} z^s \big[(z+x)^l Y(Y(a, x)b, z)u \big].\end{split}
\end{gather*}
The powers of $z$ in the last expression are bounded below, so the power series $(x+z)^{r-l}$ only contributes finitely many terms to the residue with respect to $x$. That is, we may truncate it as
\begin{gather*} p(x,z) = \sum_{t=0}^m \binom{r-l}{t} x^{r-l-t} z^t.\end{gather*}
We then obtain
\begin{gather*} a_r b_s u = F(i,j,k) \Res_{x} \Res_{z} p(x,z) z^s (z+x)^l Y(Y(a, x)b, z)u ,\end{gather*}
which immediately yields the answer we want.
\end{proof}

\begin{defn}Let $V$ be a M\"obius abelian intertwining algebra over $R$ associated to the data $(A,F,\Omega)$. Then an invariant bilinear form on $V$ is an inner product such that $(u_n v,w) = (v, e(k/2) \sum\limits_{i \geq 0} \big(L_1^{(i)} u\big)_{2k-i-n-2}w)$ whenever $u$ has weight $k$, and $(u,v) = 0$ if $u \in V^a$ and $v \in V^b$ for $a+b \neq 0$.
\end{defn}

\begin{lem}Let $V$ be a well-graded M\"obius abelian intertwining algebra over $R$ associated to the data $(A,F,\Omega)$, and suppose $V$ admits an invariant bilinear form. Then this form is self-dual if and only if for each $a \in A$ the form induces $V^0$-module isomorphisms $V^a \cong (V^{-a})^\vee$.
\end{lem}
\begin{proof}Any invariant form induces a $V^0$-module map $\phi_a\colon V^a \to (V^{-a})^\vee$ by setting $\phi_a(v)(w) = (v,w)$, because compatibility with the $V^0$-action follows from the defining relation applied to homogeneous vectors $u \in V^0$. By the orthogonality of $V^a$ and $V^b$ for $a+b \neq 0$, self-duality is then equivalent to~$\phi_a$ being an isomorphism for all~$a$.
\end{proof}

The following two results will form a primary engine behind our construction.

\begin{prop} \label{prop:existence-of-R-form-of-abelian-intertwining-algebra} Let $R$ be a subring of $\bC$ containing $1/N$ and ${\rm e}^{2\pi {\rm i}/N}$. Let $V$ be an abelian intertwining algebra over $\bC$ associated to the data $(A,F,\Omega)$, equipped with an invariant bilinear form. Suppose the following properties hold:
\begin{enumerate}\itemsep=0pt
\item[$1.$] $V$ is self-dual with respect to the invariant form.
\item[$2.$] Each $V^a$ is an irreducible $V^0$-module.
\item[$3.$] $V$ is generated by abelian intertwining subalgebras $\big\{V^{A_i}\big\}_{i \in I}$, where $A_i$ range over a set of subgroups of $A$ that generate $A$, and $V^{A_i} = \bigoplus_{a \in A_i} V^a$.
\item[$4.$] We are given M\"obius $R$-forms $V^{A_i}_R$ of each $V^{A_i}$, which coincide on pairwise intersections of the subalgebras $V^{A_i}$.
\end{enumerate}
Then the abelian intertwining subalgebra $V_R$ of $V$ over $R$ generated by $\big\{ V^{A_i}_R \big\}_{i \in I}$ is a M\"obius $R$-form $V_R$ of~$V$, i.e., $V_R \otimes_{R} \bC \cong V$.
\end{prop}
\begin{proof}By induction, it suffices to consider the case where $I = \{1,2\}$, and $A_1 \cap A_2 = \{0\}$. As it happens, this is the only case that we will use in this paper.

We define $V_R$ to be the sub-$R$-module of $V$ generated by products of elements in $V^{A_i}_R$. We then define $V^a_R$ to be the part of $V_R$ of degree $a$. In the course of this proof, we will show that if $a \in A_i$, then $V^a_R$ is equal to the degree $a$ part of $V^{A_i}_R$, but before we prove this, we will write~$\big(V^{A_i}_R\big)^a$ for the latter space to distinguish them.

To show that $V_R$ is an $R$-form of $V$, we fix $a \in A$ and consider the base change map $\phi\colon V^a_R \otimes_{R} \bC \to V^a$. It suffices to show that $\phi$ is an isomorphism. The map $\phi$ is surjective, because we may write any element of $V^a$ as a $\bC$-linear combination of products of elements of~$V^{A_i}$, and these elements in turn are $\bC$-linear combinations of elements of $V^{A_i}_R$.

To show that $\phi$ is injective, it suffices to show that any $R$-linearly independent set in $V^a_R$ is $\bC$-linearly independent. We first show that any element of $V_R$ is an $R$-linear combination of products $u_r v$, where $u \in V^{A_1}_R$ and $v \in V^{A_2}_R$. By skew-symmetry as in Lemma~\ref{lem:skew-symmetry-for-AIA}, $v_s u$ is an $R$-linear combination of elements of the desired form, so we may switch $A_1$ and $A_2$. By induction on word length, it suffices to show that any length 3 words of the form $a_r b_s w$ and $(a_r b)_s w$ have this form, where $a,b,w \in V^{A_1}_R \cup V^{A_2}_R$. Again by skew-symmetry, we do not need to consider the case $(a_r b)_s w$ separately. Under this reduction, if $b$ and $w$ are both in either $V^{A_1}_R$ or $V^{A_2}_R$, then there is nothing to show. We therefore are reduced to considering $a_r b_s w$ in the following cases:
\begin{enumerate}\itemsep=0pt
\item[1)] $a,b \in V^{A_1}_R$ and $w \in V^{A_2}_R$,
\item[2)] $b \in V^{A_1}_R$ and $a, w \in V^{A_2}_R$.
\end{enumerate}
These cases are taken to each other by applying skew-symmetry to $b$ and $w$ and swit\-ching~$A_1$ with~$A_2$. It therefore suffices to consider the first case, which is handled by Lemma~\ref{lem:shorten-words}. Thus, $V_R$ is $R$-spanned by products of the form $u_r v$, where $u \in V^{A_1}_R$ and $v \in V^{A_2}_R$. It follows immediately that $V^a_R = \big(V^{A_i}_R\big)^a$ whenever $a \in A_i$.

To show that any $R$-linearly independent set in $V^a_R$ is $\bC$-linearly independent, we consider the contrapositive. Take a finite subset $\big\{ u^{1,1}_{r_1}u^{2,1},\ldots, u^{1,k}_{r_k} u^{2,k} \big\} \subset V^a_R$, and suppose $c_1 u^{1,1}_{r_1}u^{2,1} + \cdots + c_k u^{1,k}_{r_k} u^{2,k} = 0$ for some coefficients $c_1,\ldots,c_k \in \bC$, not all zero. We want to show that $r_1 u^{1,1}_{r_1}u^{2,1} + \cdots + r_k u^{1,k}_{r_k} u^{2,k} = 0$ for some $r_1,\ldots,r_k \in R$, not all zero. We may assume that this set of vectors is minimal with respect to $\bC$-linear dependence, and in particular, that $\big\{ u^{1,1}_{r_1}u^{2,1},\ldots, u^{1,k-1}_{r_{k-1}} u^{2,k-1} \big\}$ is a $\bC$-linearly independent set. By rescaling, we may assume
\begin{gather*} u^{1,k}_{r_k} u^{2,k} = c_1 u^{1,1}_{r_1}u^{2,1} + \cdots + c_{k-1} u^{1,k-1}_{r_{k-1}} u^{2,k-1} \end{gather*}
for uniquely defined $c_1,\ldots, c_{k-1} \in \bC$.

Proposition~11.9 of~\cite{DL93} asserts that if $w^1$ and $w^2$ are nonzero elements of irreducible mo\-dules~$M_1$ and~$M_2$ of a generalized vertex algebra, and $\cY$ is an intertwining operator, then $\cY\big(w^1,z\big)w^2 = 0$ implies~$\cY$ is identically zero. Here, our hypothesis is that $V^a$ and $V^{-a_1}$ are irreducible $V^0$-modules, and the self-duality of the invariant bilinear form on $V$ implies the multiplication operation is nonzero. Thus, for any nonzero $v^1 \in V^{-a_1}_R$, $Y\big(v^1,z\big)$ is an injection to $z^s V^{a_2}((z))$ for some $s \in \bQ$. Thus,
\begin{gather*} \big\{ Y\big(v^1,z\big) u^{1,1}_{r_1}u^{2,1}, \ldots, Y\big(v^1,z\big) u^{1,k-1}_{r_{k-1}} u^{2,k-1} \big\} \end{gather*}
is a $\bC$-linearly independent set in $z^s V^{a_2}((z))$, so
\begin{gather*} Y\big(v^1,z\big) u^{1,k}_{r_k} u^{2,k} = b_1 Y\big(v^1,z\big) u^{1,1}_{r_1}u^{2,1} + \cdots + b_{k-1} Y\big(v^1,z\big) u^{1,k-1}_{r_{k-1}} u^{2,k-1} \end{gather*}
holds if and only if $b_i = c_i$ for all $i \in \{1,\ldots,k-1\}$. However, both $Y\big(v^1,z\big) u^{1,k}_{r_k} u^{2,k}$ and each summand $Y\big(v^1,z\big) u^{1,i}_{r_i}u^{2,i}$ lie in $z^sV^{a_2}_R((z))$, and by hypothesis, this space is an $R$-form of some subspace of $z^sV^{a_2}((z))$. We conclude that all $c_i$ lie in the fraction field of $R$. By clearing denominators, we see that there is some nonzero $x \in R$ such that
\begin{gather*} xc_1 u^{1,1}_{r_1}u^{2,1} + \cdots + xc_{k-1} u^{1,k-1}_{r_{k-1}} u^{2,k-1} - x u^{1,k}_{r_k} u^{2,k} = 0 \end{gather*}
and in particular, the set $\big\{ u^{1,1}_{r_1}u^{2,1},\ldots, u^{1,k}_{r_k} u^{2,k} \big\}$ is $R$-linearly dependent. Thus, $\phi$ is injective, and $V_R \otimes_{R} \bC \cong V$ as abelian intertwining algebras.

The M\"obius structure on $V_R$ follows from the formula describing the action of $L_1^{(n)}$ on products of generators.
\end{proof}

We now refine the previous result, showing that $V_R$ is self-dual and well-graded.

\begin{prop} \label{prop:extension-of-self-dual-form}Let $R$ be a subring of $\bC$ containing $1/N$ and ${\rm e}^{2\pi {\rm i}/N}$. Let $V$ be an abelian intertwining algebra over $\bC$ associated to the data $(A,F,\Omega)$, equipped with an invariant bilinear form. Suppose the following properties hold:
\begin{enumerate}\itemsep=0pt
\item[$1.$] $V$ is self-dual with respect to the invariant form.
\item[$2.$] Each $V^a$ is an irreducible $V^0$-module.
\item[$3.$] $V$ is generated by abelian intertwining subalgebras $\big\{V^{A_i}\big\}_{i \in I}$, where $A_i$ range over a set of subgroups of $A$ that generate $A$, and $V^{A_i} = \bigoplus_{a \in A_i} V^a$.
\item[$4.$] We are given M\"obius $R$-forms $V^{A_i}_R$ of each $V^{A_i}$ such that the restriction of the invariant bilinear form takes values in $R$, and $V^{A_i}_R$ is self-dual with respect to this bilinear form.
\item[$5.$] The $R$-forms and invariant bilinear forms coincide on pairwise intersections of the subalgebras $V^{A_i}$.
\end{enumerate}
Then, the following holds:
\begin{enumerate}\itemsep=0pt
\item[$1.$] The invariant bilinear form on $V$ restricts to an $R$-valued invariant bilinear form on the $R$-form $V_R$ of $V$ given in Proposition~{\rm \ref{prop:existence-of-R-form-of-abelian-intertwining-algebra}}.
\item[$2.$] $V_R$ is self-dual with respect to this bilinear form.
\item[$3.$] $V_R$ is the unique $R$-form of $V$ whose intersection with $V^{A_i}$ is $V^{A_i}_R$, such that the invariant bilinear form on $V$ restricts to an $R$-valued invariant bilinear form.
\item[$4.$] $V_R$ is well-graded as an $R$-module, i.e., the graded components $(V_R)^a_n$ are finite projective $R$-modules.
\end{enumerate}
\end{prop}
\begin{proof}
We begin by considering the invariant inner product. For $u^1_r u^2 \in V^a_R$ and $v^1_s v^2 \in V^{-a}_R$, we apply Lemma \ref{lem:shorten-words} to see that the inner product is given by
\begin{gather*}
\big(v^1_s v^2, u^1_r u^2\big) = \bigg(v^2, e(k/2) \sum_{i \geq 0} \big(L_1^{(i)} v^1\big)_{2k-i-s-2} u^1_r u^2\bigg) \\
\hphantom{\big(v^1_s v^2, u^1_r u^2\big)}{} = e(k/2) \sum_{i \geq 0} \sum_{t=0}^m \sum_{n \in \bZ_{\geq 0}} \binom{2k-i-s-2-l}{t} \times \\
\hphantom{\big(v^1_s v^2, u^1_r u^2\big)=}{} \times \binom{l}{n} \big(v^2, \big(\big(L_1^{(i)} v^1\big)_{2k-i-s-2-l-t+n}u^1\big)_{r+l+t-n} u^2\big),
\end{gather*}
which is an $R$-linear combination of inner products of vectors in $V^{A_2}_R$. By hypothesis, this is an element of $R$, so the inner product on $V_R$ is $R$-valued. We therefore find that for each $a \in A$ and $r \in \bQ$, there is a canonical map $\big(V^a_R\big)_r \to \Hom_R\big(\big(V^{-a}_R\big)_r, R\big)$ of $R$-modules induced by the inner product, and by self-duality of the base change to~$\bC$, this map is injective. To show that~$V_R$ is self-dual, it suffices to show that this map is surjective.

Let $f\colon \big(V^{-a}_R\big)_r \to R$ be an $R$-module map. Because $V$ is self-dual, $f$ is induced by taking the inner product with a unique element $u \in V^a_r$. We shall show that $u \in \big(V^a_R\big)_r$.

By precomposing $f$ with any $\big(v^1_t\big)^*$ for homogeneous $v^1 \in V^{-a_1}_R$, we obtain an $R$-linear map from some homogeneous piece $\big(V^{-a_2}_R\big)_{r+\wt(v^1)-t-1}$ to~$R$. By self-duality of $V^{A_2}_R$, this is necessarily given by the inner product with a homogeneous vector $v' \in \big(V^{a_2}_R\big)_{r+\wt(v^1)-t-1}$. That is,
\begin{gather*} \big(v^1_t u, -\big) = \big(u, \big(v^1_t\big)^* -\big) = f\big(\big(v^1_t\big)^* -\big) = (v',-). \end{gather*}
In other words, all $v^1_t$ operators take $u$ to elements of $V^{a_2}_R$. This implies that for any $v^1 \in V^{-a_1}_R$, $Y\big(v^1,z\big)u \in z^s V^{a_2}_R((z))$ for some $s \in \bQ$. By the skew-symmetry Lemma \ref{lem:skew-symmetry-for-AIA}, $Y(u,z)$ takes elements of $V^{-a_1}_R$ to elements of $z^sV^{a_2}_R((z))$.

Because $V^{A_1}_R$ is self-dual, we may present $\unit$ as a finite $R$-linear combination of products of the form $v^1_k u^1$, where $u^1 \in V^{a_1}_R$ and $v^1 \in V^{-a_1}_R$. Then $u_{-1} \unit = u$ is a finite $R$-linear combination of products of the form $u_{-1} v^1_k u^1$. By Lemma~\ref{lem:shorten-words}, this is an $R$-linear combination of products of the form $\big(u_s v^1\big)_t u^1$. Since each $u_s v^1 \in V^{a_2}_R$, we find that $u \in V^a_R$. Thus, the inner product on~$V_R$ is self-dual.

Self-duality implies each of the graded pieces $(V^a_R)_r$ are finite projective $R$-modules, because finite projective $R$-modules are precisely the dualizable $R$-modules. That is, $V_R$ is well-graded as an $R$-module. Furthermore (as an anonymous referee has pointed out), self-duality of $V_R$ implies uniqueness of $R$-forms containing the subalgebras $V^{A_i}_R$ with $R$-valued inner product. This is because $V_R$ is generated by those subalgebras, so any other such form must strictly contain $V_R$, contradicting self-duality of $V_R$.
\end{proof}

\begin{cor} \label{cor:voa-generated-by-subvoas}Let $A$ be a finite abelian group, and let $A_1$, $A_2$ be subgroups that generate~$A$. Let $n$ be the exponent of~$A$, and let $R$ be a subring of $\bC$ containing $1/n$ and $e(1/2n)$. Let~$V$ be a simple vertex operator algebra over $\bC$, and suppose $V$ admits an $A$-grading $($i.e., given by an action of the Pontryagin dual group~$A^*)$. If we are given self-dual $R$-forms $V_R^1$ and $V_R^2$ of the vertex operator subalgebras of $V$ given by the parts graded by~$A_1$ and~$A_2$, such that the $A_1 \cap A_2$-graded subalgebras of~$V_1$ and~$V_2$ are isomorphic, then there exists a unique $R$-form of $V$ whose $A_1$-graded subalgebra is $V_1$ and whose $A_2$-graded subalgebra is~$V_2$.
\end{cor}
\begin{proof}
By Theorem 3 of \cite{DM94a} our assumption that $V$ is simple implies the graded pieces of $V$ are simple $V^0$-modules. Then the hypotheses of Proposition \ref{prop:extension-of-self-dual-form} are satisfied, where $(F,\Omega)$ is trivial.
\end{proof}

\subsection{The standard form for a lattice vertex operator algebra}

A description of an integral form for a lattice vertex algebra is given in the original paper \cite{B86} where vertex algebras are defined, and more properties are established in \cite{B99}. We will follow the treatments in \cite{DG12} and \cite{M14}, because the proofs are somewhat more detailed.

Let $L$ be an even integral lattice, i.e., a finite rank free abelian group equipped with a $\bZ$-valued bilinear form that is even on the diagonal. Then there is a double cover $\hat{L}$, written as a nontrivial central extension by $ \langle \kappa \rangle \cong \{\pm 1\}$, that is unique up to non-unique isomorphism. Choosing lifts $\{ e_a \}_{a\in L}$ of lattice vectors, the cocycle defining the central extension is determined up to equivalence by the signs relating $e_a e_b$ to $e_b e_a$, and in our case it is by $e_a e_b = (-1)^{(a,b)} e_b e_a$. The twisted group ring $\bC\{L\}$ is then the quotient of $\bC\big[\hat{L}\big]$ by the ideal $(\kappa + 1)$, and we write~$\iota(e_a)$ for the image of $e_a$. The vertex algebra $V_L$ is given by the tensor product of $\bC\{L\}$ with a~Heisenberg vertex algebra $M(1)$ which we will not describe further.

If we choose a basis $\{ \gamma_1,\ldots,\gamma_d\}$ of $L$, then for any $\alpha \in L^\vee$, we set
\begin{gather*} E^-(-\alpha,z) = \exp \left( \sum_{n > 0} \frac{\alpha(-n)}{n}z^n \right) = \sum_{n \geq 0} s_{\alpha, n} z^n. \end{gather*}
Note that $E^-(-\alpha,z)E^-(\alpha,z) = 1$ for all $\alpha \in L$, i.e., $\sum\limits_{i=0}^n s_{\alpha,i} s_{-\alpha,n-i} = \delta_{n,0}$ for all $n \geq 0$. Define $(V_L)_{\bZ}$ to be the $\bZ$-span of $s_{\alpha_1,n_1} \cdots s_{\alpha_k,n_k} {\rm e}^\alpha$ for $\alpha_i \in \{ \gamma_1,\ldots,\gamma_d\}$, $n_1 \geq \cdots \geq n_k$, $k \geq 0$, and $\alpha \in L$. Here, ${\rm e}^\alpha$ denotes the image of $\iota(e_\alpha)$ for some lift $e_\alpha$ -- this choice of notation has an ambiguous sign, but in this paper we will not do calculations where the choice of sign is important. By Proposition~3.6 of~\cite{DG12}, $(V_L)_{\bZ}$ is an integral form of $V_L$, generated by ${\rm e}^{\pm\gamma_i}$, and if~$L$ is positive definite and unimodular, then $(V_L)_{\bZ}$ is a direct sum of positive definite unimodular lattices under its usual invariant bilinear form. By Proposition~5.8 of~\cite{M14}, there is a conformal element $\omega$ in an integral form for~$V_L$ if and only if~$L$ is unimodular, and the central charge is equal to the rank.

\begin{defn}Let $L$ be a positive definite even unimodular lattice, and let $R$ be a commutative ring. We call the base change $(V_L)_{\bZ} \otimes_{\bZ} R$ the standard $R$-form of~$V_L$, and denote it by~$(V_L)_R$.
\end{defn}

\begin{lem} \label{lem:standard-R-form}Let $L$ be a positive definite even unimodular lattice, and let $R$ be a commutative ring. The standard $R$-form $(V_L)_R$ is a M\"obius vertex operator algebra over $R$, and admits an invariant bilinear form for which it is self-dual.
\end{lem}
\begin{proof}The M\"obius claim is a special case of Lemma~5.6 of~\cite{B99}, and in fact Borcherds proves the claim for an action of an integral form of the universal enveloping algebra of the Virasoro algebra. The existence of the bilinear form is asserted in~\cite{B86}, and self-duality is proved as Proposition~3.6 of~\cite{DG12}.
\end{proof}

We describe some finer details of automorphisms.

\begin{prop} \label{prop:automorphism-group-of-standard-R-form}Let $L$ be a positive definite even unimodular lattice of rank $d$ with no roots $($i.e., no vectors of norm~$2)$. The contravariant functor that sends an affine scheme $\Spec R$ to the automorphism group of $(V_L)_R$ is represented by a finite type affine group scheme $\uAut (V_L)_\bZ$ over $\bZ$. This group scheme has the form $T O\big(\hat{L}\big)$, where $T$ is a normal subgroup scheme that is a split torus over $\bZ$, and $O\big(\hat{L}\big)$ is the finite flat group scheme of isometries of the double cover of $L$, and lies in a canonical exact sequence
\begin{gather*} 1 \to \uHom(L,\mu_2) \to O\big(\hat{L}\big) \to O(L) \to 1. \end{gather*}
The torus $T$ is given as the diagonalizable group $D(L) = \Spec \bZ[L]$, isomorphic to $\bG_m^d$. The intersection between $T$ and $O\big(\hat{L}\big)$ is $\uHom(L,\mu_2)$, which is isomorphic to the constant group scheme $\{\pm 1\}^d$ over any field of characteristic not equal to~$2$.
\end{prop}
\begin{proof}In the beginning of Section~2 of~\cite{DG01}, we have a description of the automorphism group of any finitely generated vertex operator algebra $V$ over $\bC$ as a closed subgroup of ${\rm GL}(U)$ for $U = \bigoplus_{m=0}^k V_m$ a generating subspace. The arguments given there extend straightforwardly to any finitely generated vertex operator algebra over a commutative ring in a way that commutes with base change, so the automorphism functor is represented by a finite type affine group scheme.

Let $\ft_R$ denote the weight 1 subspace of $(V_L)_R$. By our assumption that $L$ has no roots, $\ft$ is isomorphic to $L \otimes R$ as an $R$-module with inner product given by $(a,b)\unit = a_1 b$. By~\cite{B86}, $\ft_R$ has a canonical Lie algebra structure given by $[a,b] = a_0 b$, and in this case it is abelian. Consider an automorphism $\sigma$ of $(V_L)_R$. Because $\sigma$ respects the weight grading, it acts as an $R$-linear isometry on $\ft_R$ that induces an isometry on the $L$-grading of~$(V_L)_R$. Thus, the isometry on~$\ft_R$ is defined over $\bZ$, i.e., the restriction is given by some $\bar{\sigma} \in O(L)$.

The group $O\big(\hat{L}\big)$ acts by permutations on the set $\{ \epsilon {\rm e}^\alpha \}_{\epsilon \in \mu_2(R), \alpha \in L}$, and this action induces vertex operator algebra automorphisms on $(V_L)_R$ that act by $O(L)$ on the $L$-grading. There exists a lift $\tau$ of $\bar{\sigma}$ to $O\big(\hat{L}\big)$, so $\psi = \sigma \tau^{-1}$ is an automorphism of $(V_L)_R$ that fixes $\ft_R$ pointwise and hence fixes the $L$-grading. Because the generators ${\rm e}^{\pm \gamma_i}$ are minimal weight in their respective $L$-graded components, $\psi$ must act on these generators by scalars. Furthermore, Lemma 2.5 in \cite{DN98} (which is proved for the case $R = \bC$ but extends without change) implies this $\psi$ is necessarily an element of the torus $\Hom(L, R^\times)$, which is the group of $R$-points of $D(L)$. The $L$-grading given by assigning degree $\alpha$ to ${\rm e}^\alpha$ endows $(V_L)_\bZ$ with a $\bZ[L]$-comodule structure, so the automorphism group contains the rank~$d$ split torus~$D(L)$. We conclude that $\uAut(V_L)_\bZ = T O\big(\hat{L}\big)$.

Because $T$ acts trivially on $\ft_R$, it is clear that the intersection of $T(R)$ with $O\big(\hat{L}\big)$ is the preimage of the identity element of~$O(L)$, and this is precisely $\Hom(L,\mu_2(R))$.
\end{proof}

\begin{lem} \label{lem:lifts-of-commuting-pairs}
Let $L$ be a positive definite even unimodular lattice of rank $d$ with no roots, and let $\bar{g}$ and $\bar{h}$ be commuting automorphisms of $L$. Then for any lifts $\hat{g}$ of $\bar{g}$ and $\hat{h}$ of $\bar{h}$ to $\Aut (V_L)_R$, we have $\hat{g}\hat{h} = c_{\hat{g},\hat{h}} \hat{h} \hat{g}$ for some $c_{\hat{g},\hat{h}} \in T(R)$. Furthermore, if $\tilde{g} = \gamma \hat{g}$ and $\tilde{h} = \delta \hat{h}$ are different lifts, where $\gamma, \delta \in T(R)$, then
\begin{gather*} c_{\tilde{g},\tilde{h}} = \frac{\bar{g} \cdot \delta}{\delta} \frac{\gamma}{\bar{h} \cdot \gamma} c_{\hat{g},\hat{h}}, \end{gather*}
where $\bar{g} \cdot \delta$ denotes the image of $\delta$ under the canonical action of $\Aut L$ on $T(R) = \Hom(L, R^\times)$.
\end{lem}
\begin{proof}
The first claim is straightforward from the description of the automorphism group of~$(V_L)_R$ given in Proposition \ref{prop:automorphism-group-of-standard-R-form}. The second claim follows from a short calculation, essentially using the fact that $\hat{g} \delta = (\bar{g} \cdot \delta) \hat{g}$.
\end{proof}

\begin{lem} \label{lem:fixed-point-free-automorphisms}Let $L$ be a positive definite even unimodular lattice of rank~$d$ with no roots, let~$\bar{g}$ be a fixed-point free automorphism of $L$ such that all nontrivial powers are also fixed-point free, and let $n$ be its order. Then the set of automorphisms $g$ of $(V_L)_R$ that map to $\bar{g}$ in $O(L)$ is a torsor under~$T(R)$, and each such lift has order $n$. Given any pair $\hat{g}$ and $\tilde{g}$ of lifts of $\bar{g}$ to $\Aut (V_L)_R$, there exists an extension $R'$ of $R$ given by adjoining finitely many roots of units, such that $\tilde{g}$ and $\hat{g}$ are conjugate in $\Aut (V_L)_{R'}$, in fact by an element of $T(R')$. In particular, if~$R^\times$ is $n$-divisible $($i.e., each unit has an $n$th root$)$, then all lifts of $\bar{g}$ are conjugate in $\Aut (V_L)_R$, and if $R$ is a subring of $\bC$ containing $e(1/2n)$, then all lifts of $\bar{g}$ in $\Aut (V_L)_\bZ$ are conjugate in $\Aut (V_L)_R$. Finally, if $R$ is a subring of $\bC$ containing $\frac{1}{n}$ and $e(1/n)$, then for any fixed lift $g$ of $\bar{g}$, $(V_L)_R$ splits into a direct sum of irreducible $(V_L)_R^g$-modules $\big\{ (V_L)_R^{g = e(k/n)} \big\}_{k=0}^{n-1}$, and the invariant bilinear form induces a homogeneous perfect pairing between $(V_L)_R^{g = e(k/n)}$ and $(V_L)_R^{g = e(-k/n)}$.
\end{lem}
\begin{proof}
The parametrization of lifts of $\bar{g}$ by a $T(R)$-torsor is given in Proposition \ref{prop:automorphism-group-of-standard-R-form}. To show that each lift has order $n$, we first show that there is a lift to $O\big(\hat{L}\big)$ with order $n$. The obstruction to the existence of an order $n$ lift of any order $n$ automorphism is described in \cite[Lemma~12.1]{B92} in the case of the Leech lattice, but the argument applies in general. Namely, if $n$ is odd, there is no obstruction, and if $n$ is even, there is an obstruction if and only if $\big(\bar{g}^{n/2}v,v\big)$ is an odd integer for some $v \in L$. Since we assume $\bar{g}^{n/2}$ is a fixed-point free element of order~2, all of its eigenvalues are~$-1$, so $\big(\bar{g}^{n/2}v,v\big) = (-v,v) \in 2\bZ$, and the obstruction vanishes.

For the general problem of conjugation, we consider lifts $\hat{g}$ and $\tilde{g}$ of $\bar{g}$, and note that they satisfy $\gamma \hat{g} = \tilde{g}$ for a unique $\gamma \in T(R)$. If there is some $\delta \in T(R')$ such that $\delta^{-1} \hat{g} \delta = \tilde{g}$ for some extension $R'$ of $R$, then by Lemma~\ref{lem:lifts-of-commuting-pairs}, we have $\delta^{-1}(\bar{g} \cdot \delta) = \gamma$. We note that by the identification $T(R) = \Hom(L,R^\times)$, we may define $\gamma$ by the values in $R^\times$ it takes on a basis of~$L$. Then, if we extend~$R$ to~$R'$ by adjoining $n$-th roots of those values, we find that $\gamma$ has an $n$-th root $\gamma' \in T(R')$. Then by setting
\begin{gather*} \delta = \gamma' (\bar{g} \cdot \gamma')^2 \cdots \big(\bar{g}^{n-1} \cdot \gamma'\big)^n \in T(R'), \end{gather*}
we find that $\delta^{-1}(\bar{g} \cdot \delta) = (\gamma')^n = \gamma$. Thus,
\begin{gather*} \delta^{-1} \hat{g} \delta = \delta^{-1} (\bar{g} \cdot \delta) \hat{g} = \gamma \hat{g} = \tilde{g}, \end{gather*}
so $\hat{g}$ and $\tilde{g}$ are conjugate by an element of $T(R')$.

For the claim about $n$-divisible $R^\times$, the identification $T(R) = \Hom(L,R^\times)$ implies $T(R)$ is also $n$-divisible. For the claim about subrings of $\bC$, we note that $\Aut(V_L)_\bZ = O\big(\hat{L}\big)$, so any discrepancy $\gamma$ of lifts of $\bar{g}$ necessarily lies in $T(\bZ) = \Hom(L,\pm 1)$. Thus, if $R$ contains $e(1/2n)$, then $T(R)$ contains all $n$-th roots of elements of $T(\bZ)$.

The last claim follows immediately from Lemma \ref{lem:decomposition-is-still-self-dual}.
\end{proof}

\begin{prop} \label{prop:centralizers-for-Leech}
Let $\Lambda$ be the Leech lattice, i.e., the unique positive definite even uni\-mo\-dular lattice of rank~$24$ with no roots, and let $\bar{g}$ be a fixed-point free automorphism of prime order~$p$ $($therefore, an element in one of the classes~$2a$, $3a$, $5a$, $7a$, $13a$ according to the notation of~{\rm \cite{GAP})}. Let $R$ be a subring of $\bC$ that contains $e(1/p)$, and let $g$ be a lift of $\bar{g}$ to $(V_\Lambda)_R$. Then $C_{\Aut (V_\Lambda)_R}(g) \cong C_{\Aut V_\Lambda}(g)$, and in particular, has the form $p^{24/(p-1)}.C_{Co_0}(\bar{g})$.
\end{prop}
\begin{proof}It suffices to show that any automorphism of $(V_\Lambda)_\bC$ that commutes with $g$ also preserves the standard $R$-form, viewed as an $R$-submodule. From the description of the automorphism group given in Proposition~\ref{prop:automorphism-group-of-standard-R-form}, it is clear that both the $\bC$-form and the $R$-form have the $O\big(\hat{\Lambda}\big)$ part of the automorphism group in common, so it suffices to show that all centralizing complex elements in the torus $T = D(\Lambda)$ are defined over~$R$.

We claim that the split torus $T$ equivariantly decomposes under $\bar{g}$ as a direct sum of $\frac{24}{p-1}$ copies of the torus $\bG_m^{p-1}$ with $\bar{g}$ acting on points as
\begin{gather*} (a_1, \ldots ,a_{p-1}) \mapsto \left(\frac{1}{a_1\cdots a_{p-1}},a_1,\ldots,a_{p-2}\right). \end{gather*}
To show this, we may use the fact that $D$ gives an involutive anti-equivalence between the category of split tori and the category of free abelian groups of finite rank (see, e.g., \cite[Expos\'e~VIII, Section~1]{SGA3}). The claim then follows from the classification of indecomposable $\bZ$-free $\bZ[\bZ/p\bZ]$-modules given in~\cite{D38} (see also \cite[Theorem~74.3]{CR62}), and in particular, the fact that only one isomorphism type is fixed-point free.

Any centralizing $\bC$-point $(a_1, \ldots, a_{p-1})$ in the torus $\bG_m^{p-1}$ is fixed by this action of $\bar{g}$, so the coordinates satisfy
\begin{gather*} a_1=\cdots = a_{p-1} = \frac{1}{a_1\cdots a_{p-1}} \in \mu_p(\bC). \end{gather*}
However, $R$ contains a full set of $p$-th roots of unity by our hypothesis, so all of the centralizing complex points in~$T$ are defined over~$R$.
\end{proof}

\subsection{The cyclic orbifold construction}

We recall that if $V$ is a simple, $C_2$-cofinite, holomorphic vertex operator algebra $V$, and $g$ is an automorphism of finite order $n$, then by Theorem 10.3 of \cite{DLM97}, there is a unique $g$-twisted $V$-module, up to isomorphism (which we will call $V(g)$), and its $L(0)$-spectrum lies in some coset of $\frac{1}{n}\bZ$ in $\bQ$. We say that $g$ is anomaly-free if this coset is $\frac{1}{n}\bZ$, and we say that $g$ is anomalous otherwise.

The construction that makes this paper possible is the following: By Theorem 5.15 of \cite{vEMS}, if $V$ is a simple, $C_2$-cofinite, holomorphic vertex operator algebra $V$ of CFT type, and $g$ is an automorphism of finite order $n$, such that the nontrivial irreducible twisted modules $V(g^i)$ have strictly positive $L(0)$-spectrum, then there is some $t \in \bZ/n\bZ$ (uniquely determined by the property that the $L(0)$-spectrum of $V(g)$ lies in $\frac{t}{n^2} + \frac{1}{n}\bZ$) and an abelian intertwining algebra structure on ${}^gV = \bigoplus_{i=0}^{n-1} V(g^i)$, graded by an abelian group $D$ that lies in an exact sequence $0 \to \bZ/n\bZ \to D \to \bZ/n\bZ \to 0$, with addition law determined by the ``add with carry'' 2-cocycle
\begin{gather*} c_{2t}(i,k) = \begin{cases} 0, & i_n + k_n < n, \\ 2t, & i_n + k_n \geq n, \end{cases} \end{gather*}
where the notation $i_n$ denotes the unique representative of $i \in \bZ/n\bZ$ in $\{0,\ldots,n-1\}$. By \textit{loc.\ cit.} Proposition~5.13, the quadratic form $q_\Delta$ on $D$ given by conformal weights is isomorphic to the discriminant form on the even lattice with Gram matrix $\left( \begin{smallmatrix} -2t_n & n \\ n & 0 \end{smallmatrix} \right)$.

Furthermore, by Theorem 5.16, if $t=0$ (i.e., $g$ is anomaly-free), then the abelian intertwining algebra ${}^gV$ is naturally graded by $D = \bZ/n\bZ \times \bZ/n\bZ$, such that $V$ is the sum of the degree $(0,i)$ pieces, and that there is a simple $C_2$-cofinite, holomorphic vertex operator algebra $V/g$ of CFT type given by the sum of the degree $(j,0)$ pieces, for $0 \leq j < n$. The natural $\bZ/n\bZ$-grading from this decomposition endows $V/g$ with a canonical automorphism $g^*$ whose order is equal to~$|g|$, such that $(V/g)/g^* \cong V$ and $g^{**} = g$. More generally, they showed that if $H$ is any order $n$ subgroup of $\bZ/n\bZ \times \bZ/n\bZ$ that is isotropic with respect to $q_\Omega$, then $\bigoplus_{a \in H} ({}^gV)^a$ is a~holomorphic $C_2$-cofinite vertex operator algebra of CFT type.

\begin{prop} \label{prop:R-form-for-abelian-intertwining-algebra}Let $V$ be a holomorphic $C_2$-cofinite vertex operator algebra of CFT type, and let $g$ be an anomaly-free automorphism of order~$n$. Suppose both $V$ and $V/g$ admit $R$-forms~$V_R$ and~$(V/g)_R$ for some subring $R \subset \bC$ containing $1/n$ and $e(1/2n)$, such that the $R$-forms coincide in $V \cap V/g$ in ${}^gV$, and both $g$ and $g^*$ are automorphisms of the respective $R$-forms. Suppose further that both $V_R$ and $(V/g)_R$ admit $R$-invariant bilinear forms for which they are self-dual, and that coincide on their intersection in ${}^gV$, and assume that $g$ and $g^*$ preserve the bilinear form. Let $G$ be the automorphism group of $V_R$ and let $G^*$ be the automorphism group of $(V/g)_R$. Then:
\begin{enumerate}\itemsep=0pt
\item[$1.$] The abelian intertwining algebra ${}^gV$ has a unique $R$-form $({}^gV)_R$ with a bilinear form extending those on $V_R$ and $(V/g)_R$, and it is self-dual with respect to this form.
\item[$2.$] The group of homogeneous automorphisms of $({}^gV)_R$ is equal to a central extension of $C_G(g)$ by $\langle g^* \rangle$, and also a central extension of $C_{G^*}(g^*)$ by $\langle g \rangle$.
\item[$3.$] For any divisor $d$ of $n$, $\bigoplus_{j=0}^{n/d-1} \bigoplus_{i=0}^{d-1} ({}^gV)_R^{dj,(n/d)i}$ is an $R$-form of $V/g^d$ that is self-dual with respect to the induced invariant inner product.
\end{enumerate}
\end{prop}
\begin{proof}The first claim follows immediately from Proposition~\ref{prop:extension-of-self-dual-form}.

For the second claim, we note that restriction yields the following commutative diagram:
\begin{gather*} \xymatrix{ \Aut({}^gV)_R \ar[r] \ar[d] & C_{\Aut V_R}(g) \ar[d] \\ C_{\Aut (V/g)_R}(g^*) \ar[r] & \Aut \big(V_R^g\big), } \end{gather*}
where the centralizers of $g$ (resp.~$g^*$) are precisely the groups of automorphisms that are compatible with the grading by eigenspaces for $g$ (resp.~$g^*$). It suffices to show that the maps out of $\Aut({}^gV)_R$ are surjective with cyclic central kernel of order $|g|$, and the argument given in the proof of Proposition~2.5.2 of~\cite{GM4} works here with minimal change.

For the third claim, the fact that $\bigoplus_{j=0}^{n/d-1} \bigoplus_{i=0}^{d-1} {}^gV^{dj,(n/d)i}$ is a holomorphic vertex operator algebra follows from the fact that the group of degrees in question is isotropic of order $n$, and the identification with $V/g^d$ is straightforward. The fact that we have an $R$-form that is self-dual follows from the corresponding claims for~$({}^gV)_R$.
\end{proof}

\section{Forms of the monster vertex operator algebra}

We now use the tools from the previous section to construct $R$-forms of $V^\natural$, as $R$ ranges over some cyclotomic $S$-integer rings.

\subsection{The Abe--Lam--Yamada method}

In~\cite{ALY17}, several constructions of $V^\natural$ were given by cyclic orbifolds of odd prime order $p$ on $V_\Lambda$, and analyzed using cyclic orbifolds of order $2p$ in order to produce a comparison with the original order~2 orbifold construction of~\cite{FLM88}. We will apply a similar method to produce actions of the monster on orbifolds over various rings.

To be specific, we consider cyclic orbifolds of $V_\Lambda$ with respect to lifts of fixed-point free isometries of $\Lambda$, such that those whose order is prime yield $V^\natural$, and those whose order is a~product of two primes yield $V_\Lambda$. The key is that we may use Proposition~\ref{prop:extension-of-self-dual-form} to produce self-dual $R$-forms of abelian intertwining algebras from those orbifolds yielding~$V_\Lambda$, and this automatically yields self-dual $R$-forms for $V^\natural$ by restriction.

We summarize the information about cyclic orbifolds over $\bC$ that we need.

\begin{lem} \label{lem:cyclic-leech-orbifolds} Let $P_0 = \{ 2,3,5,7,13\}$. Then:
\begin{enumerate}\itemsep=0pt
\item[$1.$] $P_0$ is the set of primes $p$ such that there exists a fixed-point free automorphism of the Leech lattice of order $p$.
\item[$2.$] For each $p \in P_0$, there is a unique conjugacy class $[\bar{g}_p]$ of fixed-point free automorphisms in~$Co_0$ of order~$p$, and there exists a unique conjugacy class $[g_p]$ of automorphisms of~$V_\Lambda$ lifting $[\bar{g}_p]$. For any representative element $g_p$, the order of $g_p$ is $p$, and we have an isomorphism $V_\Lambda/g_p \cong V^\natural$ of vertex operator algebras over~$\bC$.
\item[$3.$] Given a pair $p_1$, $p_2$ of distinct elements of $P_0$, if there exists an automorphism of $\Lambda$ of order~$p_1p_2$, then there is a unique algebraic conjugacy class~$[\bar{g}_{p_1p_2}]$ of automorphisms of~$\Lambda$ such that $\bar{g}_{p_1p_2}^{p_1} \in [\bar{g}_{p_2}]$ and $\bar{g}_{p_1p_2}^{p_2} \in [\bar{g}_{p_1}]$. When such an automorphism exists, it is fixed-point free, and there exists a unique algebraic conjugacy class $[g_{p_1p_2}]$ of automorphisms of $V_\Lambda$ lifting $[\bar{g}_{p_1p_2}]$. For any representative element $g_{p_1p_2}$, we have an isomorphism $V_\Lambda/g_{p_1p_2} \cong V_\Lambda$ of vertex operator algebras over $\bC$.
\end{enumerate}
\end{lem}
\begin{proof}All of the claims about automorphisms of $\Lambda$ can be checked by examination of characters and power maps of $Co_0$ in \cite{GAP}. As it happens, each $[\bar{g}_p]$ is labeled $p$a with frame shape $1^{-24/p-1}p^{24/p-1}$, and each $[\bar{g}_{p_1p_2}]$ is labeled $p_1p_2$a with frame shape $1^kp_1^{-k}p_2^{-k}(p_1p_2)^k$, with the exception of order 39, where (39a, 39b) is an algebraically conjugate pair satisfying our criteria. We note that the inadmissible pairs are those satisfying $p_1p_2 \in \{ 65, 91\}$.

The claims about existence and uniqueness of lifts of automorphisms to $V_\Lambda$ follow from Section~4.2 of~\cite{LS17}, and a short version of the argument is Proposition~2.1 in~\cite{C17}. The identification of prime order cyclic orbifolds with $V^\natural$ follows from the main construction of \cite{FLM88} for $p=2$, Theorem~1.1 of~\cite{CLS16} for $p=3$, and Theorem~4.4 in~\cite{ALY17} for $p=5,7,13$. The identification of order~$p_1p_2$ cyclic orbifolds with $V_\Lambda$ follows from Theorem~4.1 in~\cite{ALY17} for $p_1p_2 \in \{6,10,14,26 \}$, and for the others, the result follows from essentially the same argument: it suffices to show that the weight 1 subspace of the irreducible twisted module $V_\Lambda(g_{p_1p_2})$ has dimension $\frac{24}{(p_1-1)(p_2-1)}$, and one can do this by manipulating the frame shape.
\end{proof}

We note that for $p \in \{3,5,7,13\}$, these orbifold constructions were conjectured in \cite{FLM88}, and partially worked out in~\cite{DM94b} and~\cite{M95}.

We now consider forms over rings. Note that in this section, we are not claiming that any particular $R$-form of $V^\natural$ necessarily carries monster symmetry.

\begin{lem} \label{lem:R-form-for-ALY-algebra}Let $P_0 = \{ 2,3,5,7,13\}$, let $p$ and $q$ be distinct elements of $P_0$ such that $pq \not\in \{65, 91\}$, and let $R = \bZ[1/pq, e(1/2pq)]$. Let $g$ be an automorphism of $V_\Lambda$ in the class~$[g_{pq}]$ described in Lemma~{\rm \ref{lem:cyclic-leech-orbifolds}} $($namely the fixed-point free class $pq\mathrm{a}$ in~{\rm \cite{GAP})} such that $g$ preserves $(V_\Lambda)_\bZ$ $($such $g$ exists by Lemma~{\rm \ref{lem:fixed-point-free-automorphisms})}. Then ${}^gV_\Lambda$ has a unique self-dual $R$-form that restricts to the standard $R$-forms on~$V_\Lambda$ and $(V_\Lambda)/g \cong V_\Lambda$.
\end{lem}
\begin{proof}Consider the decomposition $\bigoplus_{i=0}^{pq-1} V^i_R$ of the standard $R$-form of $V_\Lambda$ into eigen-$R$-modules for $g$. This decomposition preserves self-duality by Lemma~\ref{lem:decomposition-is-still-self-dual}. By Lemma \ref{lem:cyclic-leech-orbifolds}, the abelian intertwining subalgebras $\bigoplus_{i=0}^{pq-1} ({}^gV_\Lambda)^{i,0}$ and $\bigoplus_{j=0}^{pq-1} ({}^gV_\Lambda)^{0,j}$ are isomorphic to $V_\Lambda$ under its decomposition $\bigoplus_{i=0}^{pq-1} V^i$ into eigenspaces for $g$. Fix embeddings $\phi, \psi\colon V_\Lambda \to {}^gV_\Lambda$ such that $\phi|_{V^i}$ is an isomorphism to $V^{i,0}$ and $\psi|_{V^j}$ is an isomorphism $V^{0,j}$, and $\phi|_{V^0} = \psi|_{V^0}$. These embeddings are unique up to composition with automorphisms of $V_\Lambda$ that commute with $g$. Then restriction to $(V_\Lambda)_R$ yields embeddings that satisfy the hypotheses of Proposition \ref{prop:extension-of-self-dual-form}. Thus, there is a unique $R$-form $({}^gV_\Lambda)_R$ for ${}^gV_\Lambda$ that extends the $R$-forms given by $\phi$ and $\psi$, and by the uniqueness of $\phi$ and $\psi$, it is the unique $R$-form that restricts to the standard $R$-form on the two copies of $V_\Lambda$. Furthermore, the proposition asserts that $({}^gV_\Lambda)_R$ admits a unique invariant bilinear form that extends the form on each copy of $V_\Lambda$, and $({}^gV_\Lambda)_R$ is self-dual under this form.
\end{proof}

\begin{prop} \label{prop:R-forms-of-moonshine}
With notation as in Lemma~{\rm \ref{lem:R-form-for-ALY-algebra}}, the abelian intertwining subalgebras
\begin{gather*} \bigoplus_{i=0}^{q-1} \bigoplus_{j=0}^{p-1} ({}^gV_\Lambda)_R^{pi,qj} \qquad \text{and} \qquad \bigoplus_{i=0}^{p-1} \bigoplus_{j=0}^{q-1} ({}^gV_\Lambda)_R^{qi,pj}\end{gather*} of $({}^gV_\Lambda)_R$ are isomorphic self-dual $R$-forms of~$V^\natural$.
\end{prop}
\begin{proof}By Lemma~\ref{lem:cyclic-leech-orbifolds}, the two abelian intertwining subalgebras $\bigoplus_{i=0}^{q-1} \bigoplus_{j=0}^{p-1} ({}^gV_\Lambda)^{pi,qj}$ and $\bigoplus_{i=0}^{p-1} \bigoplus_{j=0}^{q-1} ({}^gV_\Lambda)^{qi,pj}$ are both isomorphic to $V^\natural$. Thus, the abelian intertwining subalgebras $\bigoplus_{i=0}^{q-1} \bigoplus_{j=0}^{p-1} ({}^gV_\Lambda)_R^{pi,qj}$ and $\bigoplus_{i=0}^{p-1} \bigoplus_{j=0}^{q-1} ({}^gV_\Lambda)_R^{qi,pj}$ are $R$-forms for $V^\natural$ that are self-dual under the induced invariant bilinear form.

We now consider the automorphism of ${}^gV_\Lambda$ given by switching coordinates, i.e., sending $({}^gV_\Lambda)^{i,j} \mapsto ({}^gV_\Lambda)^{j,i}$. Explicitly, the map is defined by taking the composite isomorphisms
\begin{gather*} ({}^gV_\Lambda)^{i,0} \overset{\phi}{\leftarrow} V^i \overset{\psi}{\to} ({}^gV_\Lambda)^{0,i} \end{gather*}
and their inverses, and extending uniquely to ${}^gV_\Lambda$ by the fact that ${}^gV_\Lambda$ is generated by the images of $\phi$ and $\psi$. This automorphism restricts to an automorphism of the $R$-form $({}^gV_\Lambda)_R$, and it transports the two $R$-forms of $V^\natural$ to each other. Thus, the two $R$-forms of $V^\natural$ are isomorphic.
\end{proof}

\begin{defn}We write $V^\natural[1/pq,e(1/2pq)]$ to denote the $R$-form of $V^\natural$ given in Proposi\-tion~\ref{prop:R-forms-of-moonshine}.
\end{defn}

\begin{rem}Because $R$ is a subring of $\bC$ containing $e(1/2pq)$, Lemma~\ref{lem:fixed-point-free-automorphisms} implies $g$ is unique up to conjugation. Thus, the formation of $V^\natural[1/pq,e(1/2pq)]$ does not depend on our choice of lift $g$.
\end{rem}

The following lemma implies the automorphism $g^*$ is compatible with the cyclic orbifold duals of $V_\Lambda$ arising from $g^p$ and $g^q$, in the sense that $(g^p)^* = (g^*)^p$ and $(g^q)^* = (g^*)^q$.

\begin{lem} \label{lem:R-forms-for-prime-order-orbifold}
With notation as in Lemma~{\rm \ref{lem:R-form-for-ALY-algebra}}, there are unique self-dual $R$-forms of~${}^{g^p}V_\Lambda$ and~${}^{g^q}V_\Lambda$ such that:
\begin{enumerate}\itemsep=0pt
\item[$1)$] ${}^{g^p}V_\Lambda^{i,0} = V^\natural[1/pq,e(1/2pq)]^{(g^*)^p = e(i/q)}$ and ${}^{g^q}V_\Lambda^{i,0} = V^\natural[1/pq,e(1/2pq)]^{(g^*)^q = e(i/p)}$ for all $i \in \bZ$,
\item[$2)$] ${}^{g^p}V_\Lambda^{0,j} = (V_\Lambda)_R^{g^p = e(j/q)}$ and ${}^{g^q}V_\Lambda^{0,j} = (V_\Lambda)_R^{g^q = e(j/p)}$ for all $j \in \bZ$.
\end{enumerate}
These $R$-forms naturally embed into $({}^g V_\Lambda)_R$ by decomposing into $g$-eigenspaces.
\end{lem}
\begin{proof}Existence follows from applying Proposition~\ref{prop:extension-of-self-dual-form} to the embeddings of $(V_\Lambda)_R$ and $V^\natural[1/pq,e(1/2pq)]$ into ${}^{g^p}V_\Lambda$ and ${}^{g^q}V_\Lambda$. Uniqueness follows from the first claim of Proposition~\ref{prop:R-form-for-abelian-intertwining-algebra}. The abelian intertwining subalgebras $\bigoplus_{i=0}^{q-1} \bigoplus_{j=0}^{pq-1} ({}^gV_\Lambda)_R^{pi,j}$ and $\bigoplus_{i=0}^{p-1} \bigoplus_{j=0}^{pq-1} ({}^gV_\Lambda)_R^{qi,j}$ of $({}^g V_\Lambda)_R$ are self-dual $R$-forms of ${}^{g^p}V_\Lambda$ and ${}^{g^q}V_\Lambda$, equipped with decompositions into $g$-eigen\-spa\-ces.
\end{proof}

\subsection{Monster symmetry}

We now consider symmetries of these $R$-forms of $V^\natural$, where once again $R = \bZ[1/pq, e(1/2pq)]$. The next lemma is where we use recent developments in finite group theory. I suspect less powerful results can yield the same answer, and I welcome any insights from specialists in finite group theory.

\begin{lem} \label{lem:maximal-subgroups-of-monster}
Let $P_0 = \{ 2,3,5,7,13\}$, let $p$ and $q$ be distinct elements of $P_0$ such that $pq \not\in \{65, 91\}$. Let $g_p$ and $g_q$ be elements of $\bM$ in classes $pB$ and $qB$, respectively $($i.e., the unique non-Fricke classes of those orders$)$. Then any subgroup of the monster simple group $\bM$ that contains $C_\bM(g_p)$ and $C_\bM(g_q)$ is $\bM$ itself.
\end{lem}
\begin{proof}This follows from known constraints on the maximal subgroups of~$\bM$, e.g., given in \cite{W17}. The important point is that for $p \in \{2,3,5\}$, $C_\bM(g_p)$ is contained in only one isomorphism type of maximal subgroup of $\bM$. However, for each prime $q$ under consideration, $C_\bM(g_q)$ contains the Sylow $q$-subgroup of~$\bM$, so it suffices to check that the order of the maximal subgroup containing~$C_\bM(g_p)$ has insufficient $q$-valuation.
\end{proof}

\begin{lem} \label{lem:action-of-centralizer-on-R-form}The $R$-form $\big({}^{g^p}V_\Lambda\big)_R$ given in Lemma~{\rm \ref{lem:R-forms-for-prime-order-orbifold}} has automorphism group given by a~central extension of $C_{\Aut V_\Lambda}(g^p)$ by $\langle (g^*)^p \rangle$. Furthermore, the abelian intertwining subalgebra $V^\natural[1/pq,e(1/2pq)]$ admits a faithful action of $C_\bM((g^*)^p)$. The same claims also hold with~$p$ and~$q$ switched.
\end{lem}
\begin{proof}The second claim of Proposition \ref{prop:R-form-for-abelian-intertwining-algebra} gives the description of the automorphism group of $\big({}^{g^p}V_\Lambda\big)_R$ as a central extension of $C_{\Aut (V_\Lambda)_R}(g^p)$ by $\langle (g^*)^p \rangle$, but in Proposition~\ref{prop:centralizers-for-Leech}, we identify this centralizer with $C_{\Aut V_\Lambda}(g^p)$. By the cyclic orbifold correspondence for prime-order orbifolds given in \cite{ALY17}, the restriction of this action to $V^\natural[1/pq,e(1/2pq)]$ induces a surjection to $C_\bM((g^*)^p)$, with kernel generated by $g^p$.
\end{proof}

\begin{thm} \label{thm:monster-form}The embedded $R$-form $V^\natural[1/pq,e(1/2pq)]$ of $V^\natural$ is preserved by the action of $\bM$ given in~{\rm \cite{FLM88}}. In particular, the automorphism group of $V^\natural[1/pq,e(1/2pq)]$ is the monster simple group~$\bM$.
\end{thm}
\begin{proof}
By Lemma \ref{lem:action-of-centralizer-on-R-form}, $V^\natural[1/pq,e(1/2pq)]$ admits faithful actions of $C_\bM(g_p)$ and $C_\bM(g_q)$. Base change to $\bC$ yields $V^\natural$, whose automorphism group is $\bM$, so we obtain embeddings of these groups in $\bM$. However, by Lemma \ref{lem:maximal-subgroups-of-monster}, these subgroups generate $\bM$. We conclude that the action of $\bM$ on $V^\natural$ preserves the $R$-form.
\end{proof}

\begin{cor}[weak modular moonshine]\label{cor:weak-modular-moonshine}
The modular moonshine conjecture holds in \linebreak $($a~slight weakening of$)$ its original form given in~{\rm \cite{R96}}. That is, for each prime $p$ dividing the order of the monster, and each element $g$ in class $p$A in the monster, there is a vertex algebra~$V^p$ over $\bF_{p^n}$ for some $n$ $($where $n=1$ is asserted in the original statement$)$ equipped with an action of~$C_\bM(g)$ such that the graded Brauer character $\sum \Tr(h|V^p)$ of a $p$-regular element $h \in C_\bM(g)$ is equal to the McKay--Thompson series $T_{gh}(\tau) = \sum \Tr(gh|V^\natural)$.
\end{cor}
\begin{proof}Following \cite{BR96}, we may take $V^p$ to be the Tate cohomology group $\hat{H}^0(g,V)$ for $V$ a form of $V^\natural$ defined over a $p$-adic integer ring. Then this is a special case of the modular moonshine conjecture proved in~\cite{BR96} and~\cite{B98}, under some assumptions that were not known to be true at the time. The assumption about homogeneous pieces of the Lie algebra $\fm \otimes \bZ_p$ was proved as Theorem~7.1 of~\cite{B99} by applying an integral enhancement of the no-ghost theorem. The last remaining open assumption is the existence of a suitable $\bZ[1/3]$-form of $V^\natural$. However, for our weakened version of Theorem~5.2 of \cite{BR96}, it suffices to have a form of $V^\natural$ with monster symmetry defined over a 2-adic integer ring, i.e., a construction that does not involve division by~2, that decomposes into $2A$-modules that are submodules of a corresponding form of $V_\Lambda$. Theorem~\ref{thm:monster-form} gives 4 separate constructions, by setting $pq \in \{ 15, 21, 35, 39 \}$, and the decomposition condition follows from the discussion in Section~5 of~\cite{BR96}. In particular, the existence of a normalizing ${\rm SL}_2(\bF_3)$ follows from the computation of normalizers of elementary abelian subgroups of the monster in~\cite{W88}. The resulting 2-adic forms are defined over unramified extensions of $\bZ_2$ of degree $4$, $6$, $12$, and $12$, respectively. Reduction mod $2$ then implies the existence of a suitable vertex algebra $V^2$ defined over $\bF_{2^n}$ for some $n \leq 4$.
\end{proof}

\begin{rem}It is possible that the method given in Section~5 of \cite{BR96} to reduce the conjectural $\bZ[1/3,e(1/3)]$-form to a $\bZ[1/3]$-form can work here to remove the superfluous roots of unity. However, we will obtain stronger results without needing this method.
\end{rem}

\subsection{Comparison of monster-symmetric forms}

Now that we have $R$-forms of $V^\natural$ with monster-symmetric self-dual invariant bilinear forms for various $R$, we will glue them. In this section, we will describe the glue.

\begin{lem} \label{lem:fixed-point-comparison}Let $p$, $q$, $r$ be distinct elements in $P_0$, and let $g \in p$B be an anomaly-free non-Fricke element of order $p$. Assume $pq, pr \not\in \{65, 91\}$. Let $R = \bZ[1/pqr, e(1/2pqr)]$. Then $\big(V^\natural[1/pq,e(1/2pq)]\big)^g_R \cong \big(V^\natural[1/pr,e(1/2pr)]\big)^g_R$.
\end{lem}
\begin{proof}By Lemma \ref{lem:R-forms-for-prime-order-orbifold}, both sides are isomorphic to $(V_\Lambda)_R^{g^*}$.
\end{proof}

\begin{lem} \label{lem:affine-subgroups-of-monster}Let $p \in \{2,3\}$. Then there exists a subgroup $H_p \subset \bM$ such that
\begin{enumerate}\itemsep=0pt
\item[$1.$] $H_p$ is isomorphic to $(\bZ/p\bZ)^2$.
\item[$2.$] All non-identity elements in $H_p$ lie in the conjugacy class~$p$B, i.e., they are non-Fricke.
\end{enumerate}
\end{lem}
\begin{proof}The existence of $p$B-pure non-cyclic elementary abelian subgroups can be extracted from the table ``maximal $p$-local subgroups'' in~\cite{ATLAS}.
\end{proof}

\begin{prop} \label{prop:isomorphisms-of-forms-for-2-and-3}Let $p \in \{2,3\}$, and let $q,r$ be distinct elements of $P_0 \setminus \{p\}$. Let $R = \bZ[1/pqr, e(1/2pqr)]$. Then
\begin{gather*} V^\natural[1/pq, e(1/2pq)] \otimes_{\bZ[1/pq, e(1/2pq)]} R \cong V^\natural[1/pr, e(1/2pr)] \otimes_{\bZ[1/pr, e(1/2pr)]} R. \end{gather*}
\end{prop}
\begin{proof}Let $H_p$ be a group of the form given in Lemma \ref{lem:affine-subgroups-of-monster}, and let $g \in H_p \setminus \{1\}$ be a nontrivial element.

Decomposing $V_{pq} = \big(V^\natural[1/pq,e(1/2pq)]\big)_R$ and $V_{pr} = \big(V^\natural[1/pr,e(1/2pr)]\big)_R$ into eigenmodules for~$H_p$, we obtain $V_{pq}^{i,j}$ and $V_{pr}^{i,j}$, for $i,j \in \{0,\ldots,p-1\}$. Because $H_p$ is generated by $p$B-elements, we have isomorphisms:
\begin{gather*}
\bigoplus_{i=0}^{p-1} V_{pq}^{0,i} \cong V_{pq}^g \cong \bigoplus_{i=0}^{p-1} V_{pq}^{i,0},\qquad
\bigoplus_{i=0}^{p-1} V_{pr}^{0,i} \cong V_{pr}^g \cong \bigoplus_{i=0}^{p-1} V_{pr}^{i,0}.
\end{gather*}
Lemma \ref{lem:fixed-point-comparison} yields $V_{pq}^g \cong V_{pr}^g$, so all of the vertex operator subalgebras are isomorphic. Both $V^\natural[1/pq, e(1/2pq)] \otimes_{\bZ[1/pq, e(1/2pq)]} R$ and $V^\natural[1/pr, e(1/2pr)] \otimes_{\bZ[1/pr, e(1/2pr)]} R$ are generated by these vertex operator subalgebras, so by the uniqueness claim of Corollary~\ref{cor:voa-generated-by-subvoas}, the two $R$-forms are isomorphic.
\end{proof}

\begin{prop} \label{prop:full-comparison-of-forms} Let $p$, $q$, $r$ be distinct elements of $P_0$, such that~$pq$ and~$pr$ are not elements of $\{ 65, 91\}$. Let $R = \bZ[1/pqr,e(1/2pqr)]$. Then
\begin{gather*} V^\natural[1/pq, e(1/2pq)] \otimes_{\bZ[1/pq, e(1/2pq)]} R \cong V^\natural[1/pr, e(1/2pr)] \otimes_{\bZ[1/pr, e(1/2pr)]} R. \end{gather*}
\end{prop}
\begin{proof}
We first note that if $p \in \{2,3\}$, then this follows from Proposition~\ref{prop:isomorphisms-of-forms-for-2-and-3}. Otherwise, we may assume $q$ is the smallest prime among $p$, $q$, $r$, and therefore, that $q \in \{2,3\}$. Let $g,h \in \bM$ lie in classes $pB$ and~$rB$, respectively. By Lemma~\ref{lem:fixed-point-comparison}, we have isomorphisms
\begin{gather*} \big(V^\natural[1/pq,e(1/2pq)]\big)^g_R \cong \big(V^\natural[1/pr,e(1/2pr)]\big)^g_R, \end{gather*}
and
\begin{gather*} \big(V^\natural[1/qr,e(1/2qr)]\big)^h_R \cong \big(V^\natural[1/pr,e(1/2pr)]\big)^h_R. \end{gather*}
However, by Proposition \ref{prop:isomorphisms-of-forms-for-2-and-3}, we have an isomorphism
\begin{gather*} \big(V^\natural[1/pq,e(1/2pq)]\big)_R \cong \big(V^\natural[1/qr,e(1/2qr)]\big)_R, \end{gather*}
hence an isomorphism
\begin{gather*} \big(V^\natural[1/pq,e(1/2pq)]\big)^h_R \cong \big(V^\natural[1/pr,e(1/2pr)]\big)^h_R. \end{gather*}
Both $\big(V^\natural[1/pq,e(1/2pq)]\big)_R$ and $\big(V^\natural[1/pr,e(1/2pr)]\big)_R$ are $R$-forms of abelian intertwining subalgebras of ${}^{gh}V_\Lambda$ generated by isomorphic $R$-forms of $\big(V^\natural\big)^g$ and $\big(V^\natural\big)^h$, so by the uniqueness claim of Corollary~\ref{cor:voa-generated-by-subvoas}, they are therefore isomorphic.
\end{proof}

\begin{rem}An alternative proof of the previous proposition can be given by using the same technique as in Proposition~\ref{prop:isomorphisms-of-forms-for-2-and-3}, because there are $p$B-pure non-cyclic elementary abelian subgroups of $\bM$ for all $p \in P_0$ (see~\cite{W88}).
\end{rem}

\subsection{Descent of self-dual forms}

Before we start gluing forms of $V^\natural$, we show that the output of gluing is unique and has monster symmetry.

\begin{lem} \label{lem:unique-R-forms-of-moonshine-from-gluing-data}Let $i_1\colon R \to R_1$, $i_2\colon R \to R_2$ be homomorphisms of subrings of $\bC$, such that $R_1 \otimes_{R} R_2$ is also a subring of $\bC$, and suppose either:
\begin{enumerate}\itemsep=0pt
\item[$1)$] $R_1$ and $R_2$ are Zariski localizations of $R$ with respect to a coprime pair of elements, or
\item[$2)$] $i_1$ and $i_2$ are faithfully flat.
\end{enumerate}
Suppose we have a gluing datum $\big(V^1, V^2, f\big)$ of self-dual forms of $V^\natural$ over the diagram $R_1 \to R_1 \otimes_{R} R_2 \leftarrow R_2$, and suppose both forms have monster symmetry. Then there is a unique form~$V$ of~$V^\natural$ over $R$ such that both~$V^1$ and $V^2$ are base-changes of $V$. Furthermore, this form has monster symmetry.
\end{lem}
\begin{proof}Because the double quotient $\bM \backslash \bM / \bM$ is a singleton, Lemma~\ref{lem:unique-voa-from-descent} asserts that we get a~unique isomorphism type of $R$-form, and the same result also shows that this form has monster symmetry.
\end{proof}

\begin{lem} \label{lem:unique-R-forms-with-monster-symmetry} Let $n \in \bZ_{\geq 1}$. Suppose we have a commutative diagram
\begin{gather*} \xymatrix{ & R \ar[rd] \ar[ld] \ar[d] \\ R_1 \ar[rd] & \cdots \ar[d] & R_n \ar[ld] \\ & T} \end{gather*}
of inclusions of commutative subrings of $\bC$, where all maps are either
\begin{enumerate}\itemsep=0pt
\item[$1)$] faithfully flat, and $R_1 \cap \cdots \cap R_n = R$, or
\item[$2)$] Zariski localizations, forming a Zariski open cover of $\Spec R$.
\end{enumerate}
Suppose we are given a self-dual $T$-form $V^\natural_T$ of $V^\natural$ with $\bM$-symmetry, and for each $i \in \{1,\ldots,n\}$, we are given a self-dual $R_i$-form $V^\natural_{R_i}$ of $V^\natural$ with $\bM$-symmetry, together with an isomorphism $V^\natural_{R_i} \otimes_{R_i} T \simto V^\natural_T$. Then there exists a self-dual $R$-form $V^\natural_R$ of $V^\natural$, unique up to isomorphism, such that base change along the diagram of inclusions yields the original diagram of forms. In particular, for each $i \in \{1, \ldots, n\}$, we have $V^\natural_{R_i} \cong V^\natural_R \otimes_{R} R_i$.
\end{lem}
\begin{proof}
The case $n=1$ is trivial, and the case $n=2$ is covered in Lemma \ref{lem:unique-R-forms-of-moonshine-from-gluing-data}. When $n \geq 3$, we apply induction, assuming existence and uniqueness for all smaller collections of rings. Then, we may reduce the question to the case $n=3$ by partitioning $\{1,\ldots,n\}$ into three nonempty sets~$X$,~$Y$,~$Z$, and introducing subrings $R_X = \bigcap_{i \in X} R_i$, $R_Y = \bigcap_{i \in Y} R_i$, and $R_Z = \bigcap_{i \in Z} R_i$ of $T$. By our induction assumption, we have unique self-dual forms $V^\natural_X$, $V^\natural_Y$, $V^\natural_Z$ of $V^\natural$ over $R_X$, $R_Y$, $R_Z$ satisfying the expected tensor product compatibility. It suffices to show that gluing to make an $R_X \cap R_Y$-form $V^\natural_{XY}$ followed by gluing with $V^\natural_Z$ yields an $R$-form $V^\natural_{XY,Z}$ that is isomorphic to the $R$-form $V^\natural_{X,YZ}$ we get by forming an $R_Y \cap R_Z$-form $V^\natural_{YZ}$ followed by gluing with $V^\natural_X$.

By uniqueness of pairwise gluing, to obtain the isomorphism $V^\natural_{XY,Z} \cong V^\natural_{X,YZ}$ it suffices to show that $V^\natural_{XY,Z} \otimes_{R} R_X \cong V^\natural_X$ and $V^\natural_{XY,Z} \otimes_{R} (R_Y \cap R_Z) \cong V^\natural_{YZ}$. The first isomorphism follows from the fact that the formation of $V^\natural_{XY,Z}$ yields isomorphisms
\begin{gather*}
V^\natural_X \cong V^\natural_{XY} \otimes_{R_X \cap R_Y} R_X \cong V^\natural_{XY,Z} \otimes_{R} (R_X \cap R_Y) \otimes_{R_X \cap R_Y} R_X \cong V^\natural_{XY,Z} \otimes_{R} R_X.
\end{gather*}
A similar argument then yields analogous expansions of $V^\natural_Y$ and $V^\natural_Z$. The second isomorphism then follows from the uniqueness statement of Lemma~\ref{lem:unique-R-forms-of-moonshine-from-gluing-data}. To elaborate, both $V^\natural_{XY,Z} \otimes_{R} (R_Y \cap R_Z)$ and $V^\natural_{YZ}$ are self-dual $R_Y \cap R_Z$-forms satisfying the property that tensor product with $R_Y$ and $R_Z$ yield $V^\natural_Y$ and $V^\natural_Z$, respectively. They are therefore isomorphic over~$R_Y \cap R_Z$.
\end{proof}

\begin{lem} \label{lem:descent-for-three-primes}Let $p$, $q$, $r$ be distinct elements of $P_0$, such that $pq$ and $pr$ are not in $\{65, 91 \}$. Then, there exists a unique self-dual $\bZ[1/pqr, e(1/2p)]$-form $V^\natural[1/pqr, e(1/2p)]$ of $V^\natural$ such that:
\begin{gather*}
V^\natural[1/pqr, e(1/2p)] \otimes_{\bZ[1/pqr, e(1/2p)]} \bZ[1/pqr, e(1/2pq)] \\
\qquad{} \cong V^\natural[1/pq, e(1/2pq)] \otimes_{\bZ[1/pq, e(1/2pq)]} \bZ[1/pqr, e(1/2pq)], \\
V^\natural[1/pqr, e(1/2p)] \otimes_{\bZ[1/pqr, e(1/2p)]} \bZ[1/pqr, e(1/2pr)] \\
\qquad{} \cong V^\natural[1/pr, e(1/2pr)] \otimes_{\bZ[1/pr, e(1/2pr)]} \bZ[1/pqr, e(1/2pr)].
\end{gather*}
Furthermore, this form has monster symmetry.
\end{lem}
\begin{proof}By Proposition~\ref{prop:full-comparison-of-forms}, we have a gluing datum for the diagram $\bZ[1/pqr, e(1/2pq)] \to \bZ[1/pqr, e(1/2pqr)] \leftarrow \bZ[1/pqr, e(1/2pr)]$ in self-dual vertex operator algebras. By Lemma~\ref{lem:unique-R-forms-with-monster-symmetry} we obtain a unique self-dual vertex operator algebra over $\bZ[1/pqr,e(1/2p)]$ such that tensor product yields the input vertex operator algebras, and furthermore, this vertex operator algebra has monster symmetry.
\end{proof}

\begin{lem} \label{lem:descent-for-tripod}Let $p$, $q$, $r$, $\ell$ be distinct elements of $P_0$ such that $pq, pr, p\ell \not\in \{65, 91\}$. Then there exists a unique self-dual $\bZ[1/p, e(1/2p)]$-form $V^\natural[1/p, e(1/2p)]$ of $V^\natural$ such that:{\samepage
\begin{gather*}
V^\natural[1/p, e(1/2p)] \otimes_{\bZ[1/p, e(1/2p)]} \bZ[1/pq, e(1/2pq)]\cong V^\natural[1/pq, e(1/2pq)], \\
V^\natural[1/p, e(1/2p)] \otimes_{\bZ[1/p, e(1/2p)]} \bZ[1/pr, e(1/2pr)]\cong V^\natural[1/pr, e(1/2pr)], \\
V^\natural[1/p, e(1/2p)] \otimes_{\bZ[1/p, e(1/2p)]} \bZ[1/p\ell, e(1/2p\ell)]\cong V^\natural[1/p\ell, e(1/2p\ell)].
\end{gather*}
Furthermore, this form has monster symmetry.}
\end{lem}
\begin{proof}We apply Lemma \ref{lem:descent-for-three-primes} for the triples $(p,q,r)$, $(p,q,\ell)$ and $(p,r,\ell)$ to obtain unique self-dual forms over $\bZ[1/pqr, e(1/2p)]$, $\bZ[1/pq\ell, e(1/2p)]$, and $\bZ[1/pr\ell, e(1/2p)]$ with monster symmetry, and they are all isomorphic to each other when base-changed to $\bZ[1/pqr\ell, e(1/2p)]$. For each pair of these forms, we apply Zariski descent (following Lemma~\ref{lem:zariski-gluing-for-voas}) to obtain unique self-dual forms $V^\natural[1/pq, e(1/2p)]$ over $\bZ[1/pq, e(1/2p)]$, $V^\natural[1/pr, e(1/2p)]$ over $\bZ[1/pr, e(1/2p)]$, and $V^\natural[1/p\ell, e(1/2p)]$ over $\bZ[1/p\ell, e(1/2p)]$, satisfying
\begin{gather*}
V^\natural[1/pq, e(1/2p)] \otimes_{\bZ[1/pq, e(1/2p)]} \bZ[1/pq, e(1/2pq)] \cong V^\natural[1/pq, e(1/2pq)], \\
V^\natural[1/pr, e(1/2p)] \otimes_{\bZ[1/pr, e(1/2p)]} \bZ[1/pr, e(1/2pr)]\cong V^\natural[1/pr, e(1/2pr)], \\
V^\natural[1/p\ell, e(1/2p)] \otimes_{\bZ[1/p\ell, e(1/2p)]} \bZ[1/p\ell, e(1/2p\ell)]\cong V^\natural[1/p\ell, e(1/2p\ell)].
\end{gather*}
Applying Zariski descent to any pair of these forms yields a self-dual $\bZ[1/p, e(1/2p)]$-form with monster symmetry, and the uniqueness claim in Lemma~\ref{lem:unique-R-forms-with-monster-symmetry} implies any pair yields an isomorphic object.
\end{proof}

\begin{lem} \label{lem:descent-for-triangle}Let $p$, $q$, $r$ be distinct elements of $P_0$, such that $pq$, $qr$, and $pr$ are not in $\{65, 91 \}$. Then, there exists a unique self-dual $\bZ[1/pqr]$-form $V^\natural[1/pqr]$ of $V^\natural$ such that
 \begin{gather*}
V^\natural[1/pqr] \otimes \bZ[e(1/2pq)] \cong V^\natural[1/pq, e(1/2pq)] \otimes_{\bZ[1/pq, e(1/2pq)]} \bZ[1/pqr, e(1/2pq)], \\
V^\natural[1/pqr] \otimes \bZ[e(1/2pr)]\cong V^\natural[1/pr, e(1/2pr)] \otimes_{\bZ[1/pr, e(1/2pr)]} \bZ[1/pqr, e(1/2pr)], \\
V^\natural[1/pqr] \otimes \bZ[e(1/2qr)]\cong V^\natural[1/qr, e(1/2qr)] \otimes_{\bZ[1/qr, e(1/2qr)]} \bZ[1/pqr, e(1/2qr)].
\end{gather*}
Furthermore, $V^\natural[1/pqr]$ has monster symmetry.
\end{lem}
\begin{proof}We apply Lemma \ref{lem:descent-for-three-primes} for the triples $(p,q,r)$, $(q,r,p)$, and $(r,p,q)$ to obtain unique self-dual forms over $\bZ[1/pqr, e(1/2p)]$, $\bZ[1/pqr, e(1/2q)]$, and $\bZ[1/pqr, e(1/2r)]$. For any pair of these forms, applying faithfully flat gluing as in Lemma \ref{lem:faithfully-flat-gluing-data-for-voas} yields a self-dual $\bZ[1/pqr]$-form satisfying the expected conditions, and uniqueness and monster symmetry follow from Lemma~\ref{lem:unique-R-forms-with-monster-symmetry}.
\end{proof}

We now come to the main theorem.

\begin{thm} \label{thm:integral-form-for-moonshine}There exists a unique self-dual $\bZ$-form $V^\natural_{\bZ}$ of the vertex operator algebra $V^\natural$ such that for any distinct $p,q \in P_0$ satisfying $pq \not\in \{65, 91\}$, we have
\begin{gather*} V^\natural_{\bZ} \otimes_{\bZ} \bZ[1/pq, e(1/2pq)] \cong V^\natural[1/pq, e(1/2pq)]. \end{gather*}
Furthermore, this integral form has monster symmetry.
\end{thm}
\begin{proof}If we let $(p,q,r)$ range over triples in $\{2,3,5,7\}$, Lemma~\ref{lem:descent-for-triangle} yields $\bZ[1/pqr]$-forms $V^\natural[1/pqr]$. Since the four rings $\bZ[1/pqr]$ form a Zariski cover of $\Spec \bZ$, we obtain a self-dual $\bZ$-form $V^\natural_\bZ$ by pairwise gluing, and Lemma~\ref{lem:unique-R-forms-with-monster-symmetry} implies monster symmetry and uniqueness with respect to isomorphisms
\begin{gather*} V^\natural_{\bZ} \otimes_{\bZ} \bZ[1/pq, e(1/2pq)] \cong V^\natural[1/pq, e(1/2pq)] \end{gather*}
as $p$, $q$ range over distinct elements of $\{2,3,5,7\}$.

We now consider the remaining cases, both of which involve the prime $13$. Let $p \in \{2,3\}$. Recall from Lemma \ref{lem:descent-for-tripod} that for any distinct $q$, $r$, $\ell$ in $P_0 \setminus \{p\}$, including the case $q=13$, we have a $\bZ[1/p,e(1/2p)]$-form $V^\natural[1/p, e(1/2p)]$ satisfying
\begin{gather*} V^\natural[1/p, e(1/2p)] \otimes_{\bZ[1/p, e(1/2p)]} \bZ[1/pq, e(1/2pq)] \cong V^\natural[1/pq, e(1/2pq)]. \end{gather*}
We claim that this form is independent of the choice of the primes $q$, $r$, $\ell$. Indeed, ap\-plying Lemma~\ref{lem:descent-for-three-primes} to all possible triples $(p,q,r)$ yields a collection of forms over the various rings $\bZ[1/pqr, e(1/2p)]$. By the tensor product compatibilities, these forms are isomorphic on Zariski intersections, so all of the $\Spec \bZ[1/p,e(1/2p)]$-forms obtained by gluing are isomorphic by uniqueness.

Thus, it suffices to show that $V^\natural_{\bZ} \otimes \bZ[1/p, e(1/2p)] \cong V^\natural[1/p, e(1/2p)]$. We have established that the isomorphism type of $V^\natural[1/p, e(1/2p)]$ is independent of the choice of $q$, $r$, $\ell$, so we may choose $q,r,\ell \in \{2,3,5,7\} \setminus \{p\}$, and then the isomorphism follows from the uniqueness of our construction of $V^\natural_{\bZ}$ in the first paragraph.
\end{proof}

\begin{cor}The inner product on $V^\natural_\bZ$ is positive definite.
\end{cor}
\begin{proof}The $\bZ[1/2]$-form of $V^\natural$ that was constructed in \cite{BR96} is positive definite, because base-change to $\bQ$ yields the positive-definite $\bQ$-form constructed in~\cite{FLM88}.

It therefore suffices to show that $V^\natural_\bZ \otimes \bZ[1/2]$ is isomorphic to the $\bZ[1/2]$-form of $V^\natural$ constructed in~\cite{BR96}. This follows from the same method as Proposition~\ref{prop:isomorphisms-of-forms-for-2-and-3}: we decompose both forms under the action of a 2B-pure 4-group $H_2$, and obtain identifications with the 2a-fixed point vertex operator subalgebra of $(V_\Lambda)_{\bZ[1/2]}$. In fact, we don't need the full isomorphism to show positive-definiteness, since it suffices to show that the corresponding $H_2$-eigenspaces are isomorphic, hence positive-definite.
\end{proof}

R. Griess has informed me that the following result is new, even without the positive definite condition.

\begin{cor}There exists a $196884$-dimensional positive-definite unimodular lattice with a~faithful monster action by orthogonal transformations.
\end{cor}
\begin{proof}The weight 2 subspace of $V^\natural_{\bZ}$ satisfies the required properties.
\end{proof}

\begin{cor}[newer modular moonshine] The stronger version of the modular moonshine conjecture, as stated in~{\rm \cite{BR96}} and~{\rm \cite{B98}}, is unconditionally true. That is, there exists a self-dual integral form $V$ of $V^\natural$ with $\bM$ symmetry, such that for each element $g$ of prime order $p$, the graded Brauer character of any $p$-regular element $h \in C_\bM(g)$ on the Tate cohomology $\hat{H}^*(g,V)$ is given by
 \begin{gather*}
\Tr(h | \hat{H}^0(g,V)) = \begin{cases} T_{gh}(\tau), & g \in p\text{A}, 3\text{C}, \vspace{1mm}\\ \dfrac{T_{gh}(\tau) + T_{gh}(\tau+1/2)}{2}, & g \in 2\text{B}, \vspace{1mm}\\ \dfrac{T_{gh}(\tau) + T_{gh\sigma}(\tau)}{2}, & g \in p\text{B}, 2|(p-1), \end{cases} \\
\Tr(h|\hat{H}^1(g,V)) = \begin{cases} 0, & g \in p\text{A}, 3\text{C}, \vspace{1mm}\\ \dfrac{T_{gh}(\tau) - T_{gh}(\tau+1/2)}{2}, & g \in 2\text{B}, \vspace{1mm}\\ \dfrac{T_{gh}(\tau) - T_{gh\sigma}(\tau)}{2}, & g \in p\text{B}, 2|(p-1), \end{cases}
\end{gather*}
where the element $\sigma$ is the unique involution in $C_\bM(g)/O_p(C_\bM(g))$ that acts as $1$ on $\hat{H}^0(g,V)$ and $-1$ on $\hat{H}^1(g,V)$, when $p \in \{3,5,7,13\}$.
\end{cor}
\begin{proof}The conditional proof given in \cite{BR96} and \cite{B98} only requires the assumptions that $V^\natural_{\bZ}$ exist (with a small technical condition at $p=2$), and that a statement about $\bZ_p$-forms of the monster Lie algebra hold. The existence of $V^\natural_{\bZ}$ is given in Theorem~\ref{thm:integral-form-for-moonshine}, and the assumption about $\bZ_p$-forms was shown to hold in \cite{B99}. The technical condition at $p=2$ appears in Section~5 of~\cite{BR96}, where Assumption~5.1 asserts that $V^\natural_\bZ \otimes \bZ[1/3]$ decomposes into graded $2.M_{12}.2$-modules that are isomorphic to submodules of $(V_\Lambda)_{\bZ[1/3]}$. We will not prove this assumption, but we note that it is only used to transfer the vanishing properties of Tate cohomology from the form of $V_\Lambda$ to the form of $V^\natural$. The proofs of vanishing only use the existence of involutions in classes 2A and 2B instead of a $2.M_{12}.2$ action. The vanishing of Tate cohomology is preserved and reflected by faithfully flat base change and also insensitive to inverting irrelevant primes, so we can check it after adding roots of unity and inverting odd primes greater than 3. For any distinct $q \in \{ 2,5,7,13\}$, Lemma \ref{lem:R-forms-for-prime-order-orbifold} gives an isomorphism from the 3B-fixed vectors in $V^\natural_\bZ \otimes \bZ[1/3q,e(1/6q)]$ to a submodule of $(V_\Lambda)_{\bZ[1/3q,e(1/6q)]}$, and the order 9 3B-pure group $H_3$ given in Lemma \ref{lem:affine-subgroups-of-monster} gives a decomposition of $V^\natural_\bZ \otimes \bZ[1/3q,e(1/6q)]$ into pieces that embed equivariantly (with respect to $C_{\bM}(3B) \cap N_{\bM}(H_3)$) into the 3B-fixed vectors. When $q$ is odd, this is sufficient to transfer the Tate cohomology vanishing properties that we need from $V_\Lambda$.
\end{proof}

We showed in Lemma \ref{lem:R-forms-for-prime-order-orbifold} that some prime order orbifold constructions of $V^\natural$ can be defined over rather small rings. In particular, if $p,q \in \{2,3,5,7,13\}$ and $pq \not\in \{65,91\}$, then we have isomorphisms
\begin{gather*} (V_\Lambda)_{\bZ[1/pq,e(1/pq)]}^\sigma \simto (V^\natural_\bZ \otimes \bZ[1/pq,e(1/pq)])^g \end{gather*}
for $\sigma$ an order $p$ lift of a fixed-point free automorphism $\bar{\sigma}$ of $\Lambda$, and $g \in \bM$ in class $pB$. Here, we show that these isomorphisms can be defined over $\bZ[1/p,e(1/p)]$. For $p=3$, this refinement was conjectured in \cite{BR96}, in the hope that such a construction could be used to construct a self-dual integral form. Thus, we are approaching this question in a somewhat backward way.

\begin{cor}
Let $p \in \{2,3,5,7\}$. Then there exists:
\begin{enumerate}\itemsep=0pt
\item[$1)$] an order $p$ automorphism $\sigma$ of $(V_\Lambda)_{\bZ[1/p,e(1/p)]}$ lifting a fixed-point free automorphism $\bar{\sigma}$ of $\Lambda$ of order $p$,
\item[$2)$] an integral form $V$ of $V^\natural$ with invariant bilinear form and monster symmetry, admitting an order $p$ automorphism $g$ in class $pB$,
\item[$3)$] an isomorphism $(V_\Lambda)_{\bZ[1/p,e(1/p)]}^\sigma \simto (V \otimes \bZ[1/p,e(1/p)])^g$ of $\bZ[1/p,e(1/p)]$-vertex algebras preserving inner products.
\end{enumerate}
\end{cor}
\begin{proof}By Lemma \ref{lem:R-forms-for-prime-order-orbifold}, we have isomorphisms
\begin{gather*} (V_\Lambda)_{\bZ[1/pq,e(1/2pq)]}^\sigma \simto (V^\natural_\bZ \otimes \bZ[1/pq,e(1/2pq)])^g \end{gather*}
for all $q \in \{2,3,5,7\} \setminus \{p\}$. We apply a descent argument following the same lines as in Theo\-rem~\ref{thm:integral-form-for-moonshine}: Faithfully flat gluing for
\begin{gather*} \bZ[1/pqr,e(1/2pq)] \to \bZ[1/pqr,e(1/2pqr)] \leftarrow \bZ[1/pqr,e(1/2pr)] \end{gather*}
yields isomorphic $\bZ[1/pqr,e(1/2p)]$-forms, as $q$ and $r$ vary over the set $\{2,3,5,7\} \setminus \{p\}$. Zariski gluing then yields isomorphic $\bZ[1/p,e(1/2p)]$-forms. When $p$ is odd, we have $\bZ[1/p,e(1/2p)] = \bZ[1/p,e(1/p)]$, and when $p=2$, the isomorphism of $\bZ[1/2]$-forms was established in~\cite{BR96}.
\end{proof}

It seems reasonable to expect that one can improve this result by removing the roots of unity.

\begin{conj}
For any $p \in \{2,3,5,7,13\}$, there is an isomorphism
\begin{gather*} (V_\Lambda)_{\bZ[1/p]}^\sigma \simto \big(V^\natural_\bZ \otimes \bZ[1/p]\big)^g \end{gather*}
of vertex operator algebras over $\bZ[1/p]$, where $\sigma$ is an order $p$ lift of a fixed-point free automorphism of $\Lambda$, and $g$ is an element of the monster in conjugacy class $p$B.
\end{conj}

Section~5 of \cite{BR96} describes how to produce a $\bZ[1/3]$-form of $V^\natural$ using a $\bZ[1/3,e(1/3)]$-form, by transporting a Galois automorphism through an ${\rm SL}_2(\bF_3)$-action on a 3B-pure elementary subgroup. This method works for $p=5$ and $p=7$, as well, but I do not know how to show that the Galois actions on $\big(V^\natural_\bZ \otimes \bZ[1/p,e(1/p)]\big)^g$ and $(V_\Lambda)_{\bZ[1/p,e(1/p)]}^\sigma$ have matching fixed-point submodules. For $p=13$, one may need to do more explicit work.

\subsection*{Acknowledgements}
I would like to thank Toshiyuki Abe for describing the constructions in \cite{ALY17} in detail at the ``VOA and related topics'' workshop at Osaka University in March 2017. I would also like to thank the anonymous referees for many helpful comments, and one referee in particular for their help with the proof of Lemma~\ref{lem:faithfully-flat-gluing-data-for-flat-modules}. This research was partly funded by JSPS Kakenhi Grant-in-Aid for Young Scientists (B) 17K14152.

\LastPageEnding

\end{document}